\documentclass[final, a4paper, 11pt]{article}

\usepackage{amsfonts}
\usepackage{amsmath}
\usepackage{amssymb}
\usepackage{amscd}
\usepackage{amsthm}

\usepackage[mathscr]{euscript}
\usepackage{color}
\usepackage{fancyhdr}
\usepackage{dsfont}
\usepackage[a4paper,scale=0.752]{geometry}
\usepackage{hyperref}

\usepackage{bm}
\usepackage{cite}
\usepackage{showkeys}

\usepackage{dblaccnt}
\usepackage{accents}
\usepackage{graphicx}
\usepackage{array}

\usepackage{hyperref}

\newtheorem{definition}{Definition}
\newtheorem{lemma}[definition]{Lemma}
\newtheorem{theorem}[definition]{Theorem}
\newtheorem{corollary}[definition]{Corollary}

\newtheorem{assumption}[definition]{Assumption}

\newtheorem{remark}[definition]{Remark}

\newcommand*{\N}{\ensuremath{\mathbb{N}}}
\newcommand*{\Z}{\ensuremath{\mathbb{Z}}}
\newcommand*{\R}{\ensuremath{\mathbb{R}}}
\newcommand*{\C}{\ensuremath{\mathbb{C}}}
\renewcommand{\i}{\mathrm{i}}
\renewcommand{\phi}{\varphi}
\renewcommand{\rho}{{\varrho}}
\renewcommand{\epsilon}{{\varepsilon}}
\newcommand{\vertiii}[1]{{\left\vert\kern-0.25ex\left\vert\kern-0.25ex\left\vert #1 
    \right\vert\kern-0.25ex\right\vert\kern-0.25ex\right\vert}}
\renewcommand{\d}[1]{\,\mathrm{d}#1 \,}

\newcommand{\D}{\mathcal{D}}
\newcommand{\I}{\mathcal{I}}
\newcommand{\A}{\mathcal{A}}

\newcommand{\J}{\mathcal{J}} 

\renewcommand{\S}{\mathcal{S}}
\newcommand{\M}{\mathcal{M}}

\renewcommand{\k}{\underline{k}}

\newcommand{\grad}{\nabla}

\newcommand{\W}{{W_{\hspace*{-1pt}\Lambda}}} 
\newcommand{\Wast}{{W_{\hspace*{-1pt}\Lambda^\ast}}} 

\newlength{\dhatheight}

\usepackage{color}
\newcommand{\high}[1] {{\color{black}{#1}}}

\begin{document}

\sloppy

\title{A High Order Numerical Method for Scattering from Locally Perturbed Periodic Surfaces}
\author{Ruming Zhang\thanks{Center for Industrial Mathematics, University of Bremen
; \texttt{rzhang@uni-bremen.de}}}
\date{}

\maketitle

\begin{abstract}
In this paper, we will introduce a high order numerical method to solve the scattering problems with non-periodic incident fields and (locally perturbed) periodic surfaces. For the problems we are considering,  the classical methods to treat quasi-periodic scattering problems no longer work, while a Bloch transform based numerical method was proposed in \cite{Lechl2017}. This numerical method, on one hand, is able to solve this kind of problems convergently; on the other hand, \high{it} takes up a lot of time and memory during the computation. The motivation of this paper is to improve this numerical method, from the regularity results of the Bloch transform of the total field, which have been studied in \cite{Zhang2017d}. As the set of the singularities of the total field is discrete in $\R$, and finite in one periodic cell, we are able to improve the numerical method by designing a proper integration contour with special conditions at the singularities. With a good choice of the transformation, we can prove that the new numerical method could possess a super algebraic convergence rate. \high{This new method improves the efficient significantly. At the end of this paper, several numerical results will be provided to show the fast convergence of the new method.} The method also provides a possibility to solve more complicated problems efficiently, e.g., three dimensional problems, or electromagnetic scattering problems.
\end{abstract}

\section{Introduction}
  
In this paper, we will propose an efficient numerical method of solving scattering problems from locally perturbed periodic surfaces with non-periodic incident fields. \high{For the special case that a quasi-periodic incident field scattered by a smooth enough periodic surface,  a classical way is to reduce the problem into one periodic cell, see \cite{Kirsc1993,Schmi2003}. Then the reduced problem could be solved by numerical methods, such as the integral equation method (see \cite{Meier2000}), and the finite element method (see \cite{Bao1995,George2011}).} However, when the incident field or the surface is no longer (quasi-)periodic, the classical method fails,  new numerical methods are required for these more difficult problems. In \cite{Lechl2017}, a Floquet-Bloch based numerical method was applied to solve this kind of problems. In this paper, the method is improved by reducing the computational  complexity significantly.

The Floquet-Bloch based methods were used in \cite{Coatl2012,Hadda2016,Lechl2016a,Lechl2016b,Lechl2017} for the numerical solutions in (locally perturbed) periodic scattering problems with non-periodic incident fields.  For theoretical results, see \cite{Lechl2015e} and \cite{Lechl2016}. \high{There is also another way to solve the locally perturbed periodic problems,  by studying the operators defined on the transfparent edge of a periodic cell.} In \cite{Joly2006,Fliss2009,Fliss2015}, a method that approximates the Dirichlet-to-Neumann map by solving an operator equation is applied, while in 
\cite{Ehrhardt2009,Ehrhardt2009a}, the problems were solved by approximating the Scattered-to-Scattered map by a so-called doubling recursive procedure.

 The method in this paper follows the results in \cite{Lechl2017,Zhang2017d}. To solve the scattering problems with a locally perturbed surface, firstly we have to transform the problem into one defined on a periodic domain. With the Bloch transform, the problem is then transformed into an equivalent  coupled system of quasi-periodic problems, which is a three dimensional problem defined in a bounded domain. The error comes from two parts in the algorithm -- the finite element method and the numerical inverse Bloch transform. As the error brought by the finite element method could be estimated by classical methods, in this paper, we will improve the numerical scheme by the study of the numerical inverse Bloch transform. To study this, a detailed behaviour of the Bloch transformed solutions is required, which has been shown in \cite{Zhang2017d}.
 
Denote by $\J_\Omega$ the Bloch transform defined on the periodic domain $\Omega$, then the total field $u$ is transformed into $w(\alpha,x)=\left(\J_\Omega u\right)(\alpha,x)$. From \cite{Zhang2017d}, the singularities in $\alpha$ locate in a discrete set in $\R$, with a square root like singularity in the neighbourhoods.  As the inverse Bloch transform has the representation of the integral on  a bounded interval, in this paper, we will pick up a new curve and let the inverse Bloch transform be defined on the curve such that the newly defined integrand satisfies some smoothness and periodic conditions. This technique to treat integrations with singularities is quite classical, and it has been described in detail and  applied to the numerical solution of the scattering problems from bounded obstacles with corners in Chap 3.5, \cite{Colto2013} (pp. 83-85). From the well known result of the convergent rule for periodic functions, see \cite{Atkin1989}, we can prove that with the new integral curve, the numerical method will converge much faster, depending on the choice of the new integration curve. 

The high order numerical method is available when the incident fields satisfy certain conditions. As far as we know, both the point sources and the Herglotz wave functions satisfy the conditions, thus the scattering problems with these incident fields  could be solved numerically by the high order method introduced in this paper. Numerical examples for these cases will be presented at the end of this paper. However, the plane waves do not satisfy the conditions, so the numerical method fails in this case.

The rest of the paper is organized as follows. In Section 2, we will review some results for Bloch transform and the locally perturbed periodic problems. In Section 3, we recall the singularity results for the Bloch transform of the total field. In Section 4, we will introduce the inverse Bloch transform defined on a new contour. In Section 5, we will study the finite element method with the new inverse Bloch transform. In Section 6, some numerical examples will be shown.

\section{Bloch transform and the locally perturbed periodic problems}\label{sec:review}

\subsection{Introduction to the Scattering problems}
Suppose $\zeta$ is a Liptschiz periodic function with period $\Lambda$, and the periodic surface is defined by
\begin{equation*}
\Gamma:=\{(x_1,\zeta(x_1)):\,x_1\in\R\}.
\end{equation*} 
Let $\zeta_p$ be a function such that ${\rm supp}(\zeta_p-\zeta)$ is compact in $\R$. For simplicity, assume that ${\rm supp}(\zeta_p-\zeta)\subset \left(-\frac{\Lambda}{2},\frac{\Lambda}{2}\right]$, then define the locally perturbed surface by
\begin{equation*}
\Gamma_p:=\{(x_1,\zeta_p(x_1)):\,x_1\in\R\}.
\end{equation*}
Define $\Lambda^*=\frac{2\pi}{\Lambda}$, and
\begin{equation*}
\W=\left(-\frac{\Lambda}{2},\frac{\Lambda}{2}\right],\quad\Wast=\left(-\frac{\Lambda^*}{2},\frac{\Lambda^*}{2}\right]=\left(-\frac{\pi}{\Lambda},\frac{\pi}{\Lambda}\right]. 
\end{equation*}
We can define the space above the surfaces $\Gamma$ and $\Gamma_p$ by
\begin{eqnarray*}
&&\Omega:=\{(x_1,x_2):\,x_1\in\R,\,x_2>\zeta(x_1)\}\\
&&\Omega_p:=\{(x_1,x_2):\,x_1\in\R,\,x_2>\zeta_p(x_1)\}.
\end{eqnarray*}
For simplicity, we assume that $\Gamma$ and $\Gamma_p$ are both above the straight line $\R\times\{0\}$, and $H$ be a positive number such that $H>\max\{\sup\{\zeta\},\sup\{\zeta_p\}\}$. Define the line $\Gamma_H=\R\times\{H\}$ and the spaces with finite height
\begin{eqnarray*}
&&\Omega^H:=\{(x_1,x_2):\,x_1\in\R,\,\zeta(x_1)<x_2<H\}\\
&&\Omega_p^H:=\{(x_1,x_2):\,x_1\in\R,\,\zeta_p(x_1)<x_2<H\}.
\end{eqnarray*}
We can also define the curves and  domains in one periodic cell:
\begin{eqnarray*}
&&\Gamma^\Lambda=\Gamma\cap\W,\quad\Gamma^\Lambda_H=\Gamma_H\cap\W;\\
&&\Omega^\Lambda=\Omega\cap \high{\W}\times \R,\quad\Omega^\Lambda_H=\Omega^H\cap \W\times\R.
\end{eqnarray*}

\begin{figure}[htb]
\centering
\includegraphics[width=0.6\textwidth]{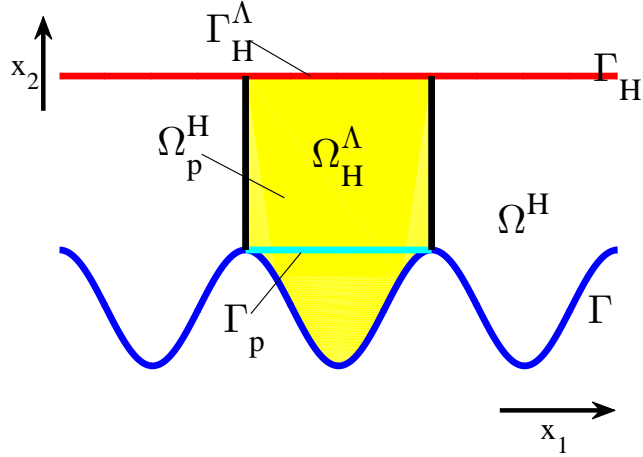} 
\caption{Locally perturbed periodic structure.}
\end{figure}

Given an incident field $u^i$, the scattering problem is to find a total field $u$ such that it satisfies the Helmholtz equation in the space $\Omega_p$ 
\begin{equation}
\Delta u+ k^2 u=0\quad \text{ in }\Omega_p
\end{equation}
and the scattered field $u^s:=u-u^i$ propagates upward, i.e.,
\begin{equation}
u^s(x)=\frac{1}{2\pi}\int_\R e^{\i x_1 \xi+\i\sqrt{k^2-|\xi|^2}(x_2-H)}\widehat{u}^s(\xi,H)\d\xi\quad x_2>H,
\end{equation}
where $\widehat{u}^s(\xi,H)$ is the Fourier transform of $u^s(\cdot,H)$. The upward propagating condition is equivalent to 
\begin{equation}
\frac{\partial u}{\partial x_2}=T^+ u|_{\Gamma_H}+\left(\frac{\partial u^i}{\partial x_2}-T^+\left[u^i|_{\Gamma_H}\right]\right)\quad\text{ on }\Gamma_H,
\end{equation}
where $T^+$ is the Dirichlet-to-Neumann map \high{(see \cite{Chand2005})} defined by
\begin{equation}
T^+ \phi(x_1)=\frac{\i}{2\pi}\int_\R\sqrt{k^2-|\xi|^2} e^{\i x_1 \xi+\i\sqrt{k^2-|\xi|^2}(x_2-H)}\widehat{\phi}(\xi)\d\xi
\end{equation}
that is bounded from $H^{1/2}_r(\Gamma_H)$ to $H^{-1/2}_r(\Gamma_H)$ for all $|r|<1$. Define the total field $u=u^i+u^s$, then we can formulate the variational problem, i.e., to find $u\in \widetilde{H}^1_r(\Omega^H_p)$ such that for any $\phi\in H^1(\Omega^H_p)$ with a compact support such that
\begin{equation}\label{eq:var_origional}
\int_{\Omega^H_p}\left[\nabla u\cdot\nabla \overline{\phi}-k^2u\overline{\phi}\right]\d x-\int_{\Gamma_H}T^+(u|_{\Gamma_H}) \overline{\phi}\d s=\int_{\Gamma_H}\left[\frac{\partial u^i}{\partial x_2}-T^+(u^i|_{\Gamma_H})\right]\overline{\phi}\d s,
\end{equation} 
where $H_r^s(\Omega^H_p)$ is a weighted space defined by
\begin{equation*}
H_r^s(\Omega^H_p):=\left\{\phi\in \mathcal{D}'(\Omega^H_p):\, (1+|x_1|^2)^{r/2}\phi(x)\in H^s(\Omega^H_p)\right\},
\end{equation*}
and $\widetilde{H}_r^s(\Omega^H_p)$ is a subspace defined by
\begin{equation*}
\widetilde{H}_r^s(\Omega^H_p)=\left\{\phi\in H_r^s(\Omega^H_p):\,\phi|_{\Gamma_p}=0\right\}.
\end{equation*}

In \cite{Chand2010}, the unique solvability of the variational form was proved.
\begin{theorem}
For any $|r|<1$, given any $u^i\in H^1_r(\Omega^H_p)$, there is a unique \high{solution} $u\in \widetilde{H}^1_r(\Omega_p^H)$ of the variational problem \eqref{eq:var_origional}.
\end{theorem}

\subsection{Bloch transform and well-posedness of the scattering problems}

Firstly, we have to introduce the Bloch transform defined on the periodic space $\Omega^H$. For any function $\phi\in C_0^\infty(\Omega^H)$, define
\begin{equation*}
\left(\J_{\Omega}\phi\right)(\alpha,x)=\sum_{j\in\Z}\phi(x_1+\Lambda j,x_2)e^{-\i\alpha\Lambda j},\quad\alpha\in\Wast;\,x\in\high{\Omega^H}.
\end{equation*}
\begin{remark}
In this paper, we denote the one dimensional Bloch transform defined on any periodic domain in $\R^2$ by $\J_\Omega$.
\end{remark}
For any fixed $\alpha\in\Wast$, the function $\left(\J_\Omega\phi\right)$ is $\alpha$-quasi-periodic with period $\Lambda$, i.e.,
\begin{equation*}
\left(\J_\Omega\phi\right)\left(\alpha,\left(\begin{matrix}
x_1+\Lambda j\\x_2
\end{matrix}\right)\right)=e^{\i\alpha\Lambda}\left(\J_\Omega\phi\right).
\end{equation*}
\high{For the positive integer $m$, let the space $H^m(\Wast;H^s(\Omega^\Lambda_H))$ be defined with the norm
\begin{equation*}
\|\phi\|_{H^m(\Wast;H^s(\Omega^\Lambda_H))}=\left[\sum_{\ell=0}^m\int_\Wast\|\partial_\alpha^\ell\phi(\alpha,\cdot)\|^2_{H^s(\Omega^\Lambda_H)}\right]^{1/2},
\end{equation*}
then extend the definition to any real number $r$ by interpolation and duality arguments.} Define the subspace $H_0^r(\Wast;H^s_\alpha(\Omega^\Lambda_H))$ with a further condition that $\phi\in H^r(\Wast;H^s(\Omega^\Lambda_H))$ is $\Lambda^*$-periodic in $\alpha$ and $\phi(\alpha,\cdot)\in H^s_\alpha(\Omega^\Lambda_H)$. The mapping property of the Bloch transform was proved in \cite{Lechl2016}.
\begin{theorem}
The Bloch transform extends to an isomorphism between $H^s_r(\Omega)$ and $H_0^r(\Wast;H^s_\alpha(\Omega^\Lambda_H))$ for any $s,\,r\in\R$. When $s=r=0$, $\J_\Omega$ is an isometry with the inverse operator 
\begin{equation}
\left(\J^{-1}_{\Omega^H}w\right)\left(\begin{matrix}
x_1+\Lambda j\\x_2
\end{matrix}\right)=\left[\frac{\Lambda}{2\pi}\right]^{1/2}\int_\Wast w(\alpha,x)e^{\i\Lambda j\alpha}\d\alpha,\quad x\in\Omega^\Lambda_H,\,j\in\Z.
\end{equation}
Moreover, the $L^2$-adjoint operator $J^*_{\Omega}$ equals to its inverse.
\end{theorem}

As the Bloch transform is defined on periodic domains, we have to transform the problem defined on a locally perturbed domain into one defined on a periodic domain. Following \cite{Lechl2016,Lechl2017}, define a  diffeomorphism $\Phi_p$ from $\Omega^H$ into $\Omega_p^H$ by
\begin{equation*}
\Phi_p(x_1,x_2)=\left(x_1,x_2+\frac{(x_2-H)^3}{(\zeta(x_1)-H)^3}(\zeta_p(x_1)-\zeta(x_1)\right),
\end{equation*}
which has a \high{support contained} in $\Omega^\Lambda_H$. \high{Then $u_T:=u\circ\Phi_p$ satisfies the following variational problem}
\begin{equation}\label{eq:var_trans}
\begin{aligned}
\int_{\Omega_H}\left[A_p\nabla u_T\cdot\nabla\overline{\phi_T}-k^2 c_p u_T\overline{\phi_T}\right]\d x&-\int_{\Gamma_H}T^+(u_T|_{\Gamma_H})\overline{\phi_T}\d s\\&=\int_{\Gamma_H}\left[\frac{\partial u^i}{\partial x_2}-T^+(u^i|_{\Gamma_H})\right]\overline{\phi_T}\d s
\end{aligned}
\end{equation}
for all $\phi_T\in H^1(\Omega^H)$ with compact support, where
\begin{eqnarray*}
&&A_p(x)=|\det\nabla\Phi_p(x)|\left[(\nabla\Phi_p(x))^{-1}((\nabla\Phi_p(x))^{-1})^T\right]\in L^\infty(\Omega^H,\R^{2\times 2}),\\
&&c_p(x)=|\det\nabla\Phi_p(x)|\in L^\infty(\Omega^H).
\end{eqnarray*}

Let $w=\J_\Omega u_T$ then $w\in L^2(\Wast;\widetilde{H}^1_\alpha(\Omega^\Lambda_H))$, and satisfies the following variational problem for any $\phi\in L^2(\Wast;\widetilde{H}^1_\alpha(\Omega^\Lambda_H))$
\begin{equation}\label{eq:var_bloch}
\begin{aligned}
\int_\Wast a_\alpha(w(\alpha,\cdot),\phi(\alpha,\cdot))\d\alpha+b(\J^{-1}_\Omega w,\J^{-1}_\Omega \phi)=\int_\Wast\int_{\Gamma^\Lambda_H}f(\alpha,\cdot)\overline{\phi}(\alpha,\cdot)\d s\d\alpha,
\end{aligned}
\end{equation}
where 
\begin{eqnarray*}
&&\begin{aligned}
a_\alpha(w(\alpha,\cdot),\phi(\alpha,\cdot))&=\int_{\Omega^\Lambda_H}\left[\nabla w(\alpha,\cdot)\cdot\nabla\overline{\phi}(\alpha,\cdot)-k^2w(\alpha,\cdot)\overline{\phi}(\alpha,\cdot)\right]\d x\\&-\int_{\Gamma_H}T^+_\alpha(w(\alpha,\cdot)|_{\Gamma_H})\overline{\phi}(\alpha,\cdot)
\end{aligned}\\
&&b(\xi,\psi)=\left[\frac{\Lambda}{2\pi}\right]^{1/2}\int_{\Omega^\Lambda_H}(A_p-I_2)\nabla\xi\cdot\nabla\overline{\psi}\d x-k^2\left[\frac{\Lambda}{2\pi}\right]^{1/2}\int_{\Omega^\Lambda_H}(c_p-1)\xi\overline{\psi}\d x,\\
&&f(\alpha,\cdot)=\frac{\partial \J_\Omega u^i(\alpha,\cdot)}{\partial x_2}-T^+_\alpha\left[(\J_\Omega u^i)(\alpha,\cdot)|_{\Gamma^\Lambda_H}\right],
\end{eqnarray*}
with $T^+_\alpha$ defined by
\begin{equation}
T^+_\alpha(\psi)=\i\sum_{j\in\Z}\sqrt{k^2-|\Lambda^* j-\alpha|^2}\widehat{\psi}(j)e^{\i(\Lambda^*j-\alpha)x_1}\quad \psi=\sum_{j\in\Z}\widehat{\psi}(j)e^{\i(\Lambda^*j-\alpha)x_1}.
\end{equation}

\high{
\begin{remark}
The system is decoupled when the surface is purely periodic. In this case, $\Phi_p$ is the identity operator, thus $A_p=I_2$ and $c_p=1$. From the definition of the term $b$, it disappears in this case, thus the system has the form of
\begin{equation*}
\int_\Wast a_\alpha(w(\alpha,\cdot),\phi(\alpha,\cdot))\d\alpha=\int_\Wast\int_{\Gamma^\Lambda_H}f(\alpha,\cdot)\overline{\phi}(\alpha,\cdot)\d s\d\alpha.
\end{equation*}
\end{remark}
}

At the end of this subsection, we will list some of  the useful results in \cite{Lechl2017}.

\begin{theorem}\label{th:uni_solv}
Assume $u^i\in H^1_r(\Omega^H_p)$ for some $r\in[0,1)$, then $u_T\in \widetilde{H}^1_r(\Omega^H)$ satisfies \eqref{eq:var_origional} if and only if $w=\J_\Omega u_T\in H_0^r(\Wast;\widetilde{H}^1_\alpha(\Omega^\Lambda_H))$ satisfies \eqref{eq:var_bloch}. 
If $\Gamma_p$ is Lipschitz, then \eqref{eq:var_bloch} is uniquely solvable in $H^r_0(\Wast;\widetilde{H}^1_\alpha(\Omega^\Lambda_H))$ for all $u^i\in H^1_r(\Omega^H_p)$, $r\in[0,1)$.
\end{theorem}

\begin{theorem}\label{th:higherregu}
If $u^i\in H^2_r(\Omega^H_p)$ for $r\in[0,1)$ and $\zeta,\zeta_p\in C^{2,1}(\R)$， then $w(\alpha,\cdot)\in \widetilde{H}^2_\alpha(\Omega^\Lambda_H)$ and $u_T=\J^{-1}_\Omega w\in \widetilde{H}^2(\Omega^H)$.
\end{theorem}

\begin{theorem}\label{th:continuous}
If $\Gamma_p$ is Lipschitz continuous and $u^i\in H^1_r(\Omega^H_p)$ for $r\in(1/2,1)$, then $w\in L^2(\Wast;\widetilde{H}^1_\alpha(\Omega^\Lambda_H))$ equivalently satisfies for all $\alpha\in\Wast$ and $\phi_\alpha\in \widetilde{H}^1_\alpha(\Omega^\Lambda_H)$ that 
\begin{equation}
a_\alpha(w(\alpha,\cdot),\phi_\alpha)+b(\J^{-1}_\Omega w,\phi_\alpha)=\int_{\Gamma^\Lambda_H}f(\alpha,\cdot)\overline{\phi_\alpha}\d s,
\end{equation}
\end{theorem}

\subsection{Numerical method}

Assume that $\M_h$ is a family of regular and quasi-uniform meshes \high{(see \cite{Brenn1994})} of $\Omega^\Lambda_H$, such that the width of each mesh is not larger than $h$. Suppose $\{\phi^{(l)}_M\}_{l=1}^M$ is the piecewise linear nodal \high{function} that equals to one at the $l$-th nodal $x^{(\ell)}=\left(x_1^{(\ell)},x_2^{(\ell)}\right)^\top$ while equals to zero \high{at other nodal points}. Then the space spanned by $\{\phi^{(l)}_M\}_{l=1}^M$ is a subset of $H^1_0(\Omega^\Lambda_H)$. 

For the domain $\Wast$, define the uniformly located grid points 
\begin{equation*}
\alpha_N^{(j)}=-\frac{\pi}{\Lambda}+\frac{2\pi j}{N\Lambda},\quad j=1,2,\dots,N,
\end{equation*}
and $\{\psi_N^{(j)}\}_{j=1}^N$ be the basis of the functions that are piecewise constant on $[\alpha^{(j)}_N-\pi/(N\Lambda),\alpha^{(j)}_N+\pi/(N\Lambda)]$. The finite element space $X_{N,h}$ is formulated by
\begin{equation*}
X_{N,h}=\left\{v_{N,h}(\alpha,x)=\sum_{j=1}^N\sum_{l=1}^M v^{(j,l)}_{N,h}\exp\left(-\i\alpha_N^{(j)}x_1\right) \psi^{(j)}_N(\alpha)\phi^{(l)}_M(x):\,v^{(j,l)}_{N,h}\in\C\right\}.
\end{equation*}
Then the numerical scheme is to seek for a finite element solution $w_{N,h}\in X_{N,h}$ such that 
\begin{equation*}
\int_\Wast a_\alpha(w_{N,h},\phi_{N,h})\d\alpha+b(\J^{-1}_\Omega w_{N,h},\overline{\J^{-1}_\Omega \phi_{N,h}})=\int_\Wast\int_{\Gamma^\Lambda_H}f(\alpha,\cdot)\overline{\phi}_{N,h}\d s\d\alpha
\end{equation*}
for all $\phi_{N,h}\in X_{N,h}$. The unique solvability and  convergence of the finite element scheme of the variational form \eqref{eq:var_bloch}.

\begin{theorem}\label{th:numeric_old}
Assume that $u^i\in H^2_r(\Omega^H_p)$ for $r\geq 1/2$, $\zeta,\zeta_p\in C^{2,1}(\R)$, then is uniquely solvable in $X_{N,h}$. The error between the numerical solution $w_{N,h}$ and exact one $w$ is bounded be
\begin{equation}
\|w_{N,h}-w\|_{L^2(\Wast;H^l_\alpha(\Omega^\Lambda_H))}\leq C h^{1-l}(N^{-r}+h)\|f\|_{H^r_0(\Wast;H^{1/2}_\alpha(\Gamma^\Lambda_H))}.
\end{equation}
\end{theorem}

\high{
\begin{remark}
For the special case that the surface is non-perturbed, the same error estimate of the numerical result is also obtained. Moreover, when $u^i$ satisfies that $\J_\Omega u^i\in W_p^{1,p}(\Wast;H_\alpha^1(\Omega^\Lambda_H))$ for some $p\in[1,2)$, then 
\begin{equation*}
\|w_{N,h}-w\|_{L^2(\Wast;H_\alpha^\ell(\Omega^\Lambda_H))}\leq C(h^{2-\ell}+N^{-1})\|f\|_{L^2(\Wast;H^{1/2}_\alpha(\Gamma^\Lambda_H))}.
\end{equation*}
The space $W_p^{1,p}(\Wast;H_\alpha^1(\Omega^\Lambda_H))$ is defined by
\begin{equation*}
\|w\|^p_{W_p^{1,p}(\Wast;H_\alpha^1(\Omega^\Lambda_H))}=\int_\Wast\left[\|w(\alpha,\cdot)\|^p_{H^1_\alpha(\Omega^\Lambda_H)}+\left\|\frac{\partial}{\partial\alpha}w(\alpha,\cdot)\right\|^p_{H^1_\alpha(\Omega^\Lambda_H)}\right]\d\alpha.
\end{equation*}
For details see \cite{Lechl2016b}.
\end{remark}
}

We will introduce the numerical method in \cite{Lechl2016b} briefly in the rest of this section. Let $w^s(\alpha,x)=w(\alpha,x)-\left(\J_\Omega u^i\right)(\alpha,x)$ and $w_0(\alpha,x)=e^{\i\alpha x_1}w^s(\alpha,x)$, then $w_0\in L^2(\Wast;H_0^1(\Omega^\Lambda_H)))$.

Recall that for the choice of $X_{N,h}$, piecewise constant basis functions are chosen for $\alpha$ in $\Wast$, we can define the  discrete space as follows. Still denote by $\left\{\phi_M^{(\ell)}\right\}_{\ell=1}^M$ the piecewise linear nodal functions in $\Omega^\Lambda_H$, and $\left\{\psi_N^{(\ell)}\right\}_{\ell=1}^N$ the piecewise constant basic function in $\Wast$. Suppose there is a positive integer $M'<M$ such that $\phi_\ell\big|_{\Gamma^\Lambda}=0$ for any $\ell=1,\dots,M'$ and $\phi_\ell\big|_{\Gamma^\Lambda}\neq0$ for any $\ell=M'+1,\dots,M.$. Thus  ${\rm span}\left\{\phi_{M'}^{(\ell)}\right\}_{\ell=1}^{M'}\subset\widetilde{H}_0^1(\Omega^\Lambda_H)$. 

Define the discrete space by
\begin{equation*}
Y^{(j)}_{N,h}=\left\{v^{(j)}_{N,h}=\sum_{\ell=1}^{M}v_{N,h}^{(j,l)}\psi_N^{(j)}(\alpha)\phi_M^{(l)}(x):\,v_{N,h}^{(j,l)}\in\C\right\}\subset L^2(\Wast;H_0^1(\Omega^\Lambda_H))
\end{equation*}
for $j=1,\dots,N$ and also the subspace of functions that vanish on $\Gamma^\Lambda$
\begin{equation*}
\widetilde{Y}^{(j)}_{N,h}=\left\{v^{(j)}_{N,h}=\sum_{\ell=1}^{M'}v_{N,h}^{(j,l)}\psi_N^{(j)}(\alpha)\phi_M^{(l)}(x):\,v_{N,h}^{(j,l)}\in\C\right\}\subset L^2(\Wast;\widetilde{H}_0^1(\Omega^\Lambda_H)).
\end{equation*}
Then setting
\begin{equation*}
Y^0_{N,h}=Y^{(1)}_{N,h}\oplus\cdots\oplus Y^{(N)}_{N,h},\quad \widetilde{Y}^0_{N,h}=\widetilde{Y}^{(1)}_{N,h}\oplus\cdots\oplus \widetilde{Y}^{(N)}_{N,h},
\end{equation*}
then for any $w_0\in X_{N,h}$, there is a unique vector $\left(w_0^{(j)}\right)_{j=1}^N\in Y^0_{N,h}$ such that $w_0=\sum_{j=1}^N \exp\left(-\i\alpha_N^{(j)}x_1\right) w^{(j)}_0$.  Then the inverse Bloch transform $\J^{-1}_\Omega$ applied to the function $w_0$ equals to the numerical inverse Bloch transform $\J^{-1}_{\Omega,N}$ applied to $\left\{e^{-\i\alpha_n^{(j)}x_1} w_0^{(j)}\left(\alpha_N^{(j)},x\right)\right\}_{j=1}^N$
\begin{equation*}
\J_\Omega^{-1} w_0=\J_{\Omega,N}^{-1}\left(\left\{e^{-\i\alpha_n^{(j)}x_1} w_0^{(j)}\left(\alpha_N^{(j)},x\right)\right\}_{j=1}^N\right)=\frac{\Lambda^*}{N}\sum_{j=1}^N e^{-\i\alpha_N^{(j)}x_1} w_0^{(j)}\left(\alpha_N^{(j)},x\right).
\end{equation*}

Suppose $w_0=\sum_{j=1}^N \exp\left(-\i\alpha_N^{(j)}x_1\right) w^{(j)}_0$ where $w_0^{(j)}=\sum_{m=1}^M w_{N,h}^{(j,m)}\psi_N^{(j)}\phi_M^{(m)}$, omitting the details in \cite{Lechl2017}, the final linear system is set up as follows.
\begin{equation}\label{eq:linearsystem}
\left(\begin{matrix}
A_1 & 0 & \cdots & 0 & C_1\\
0   & A_2 & \cdots & 0 & C_2\\
\vdots & \vdots & \vdots & \vdots & \vdots\\
0 & 0& \cdot & A_N & C_N\\
B_1 & B_2 & \cdot& B_N & I_M
\end{matrix}
\right)\left(\begin{matrix}
W_1\\ W_2\\ \vdots \\ W_N\\ U
\end{matrix}
\right)=\left(
\begin{matrix}
F_1\\ F_2\\ \vdots\\ F_N\\ 0
\end{matrix}
\right),
\end{equation}
where $A_j,\,B_j,\,C_j,\,I_m$ are $M\times M$ matrices and $F_j$ are $M\times 1$ vertices, $W_j=\left(w_{N,h}^{(j,1)},\dots,w_{N,h}^{(j,m)}\right)^\top$ and $U=\left(u_h^{(1)},\dots,u_h^{(M)}\right)^\top$, where
\begin{equation*}
u_h^{(\ell)}=\frac{\Lambda^*}{N}\sum_{j=1}^N e^{-\i \alpha_N^{(j)}x_1^{(\ell)}}w_{N,h}^{(j,\ell)}.
\end{equation*}

The matrix linear system \eqref{eq:linearsystem} is a $M(N+1)\times M(N+1)$-sparse matrix with $3N+1$ non-zero $M\times M$ blocks. In \cite{Lechl2017},  a GMRES iteration method with an incomplete LU-decomposition based pre-conditioner was applied to the solution of the linear system. The convergent rate with respect to $N$ is $N^{-r}$ where $1/2\leq r<1$, which means that if the mesh size $h$ is small enough, a small error requires a relatively large $N$. However, if $N$ gets larger, there are several disadvantages to solve the linear system \eqref{eq:linearsystem}.
\begin{itemize}
\item \high{Setting up the matrix takes up a lot of time and memory.} 
\item The time of solving the system is long.
\end{itemize}
To solve the scattering problems more efficiently, we would like to look for a new method that converges faster in $N$.

\section{Regularity results for the scattering problems}

A natural way to look for a new numerical method that converges faster as $N$ tends to $\infty$ is to go deeper into the regularity of the field $w(\alpha,\cdot)$ with respect to $\alpha\in\Wast$. Theorem a in \cite{Kirsc1993} gave a result for non-perturbed surfaces with incident plane waves. It could be a good inspiration for the problems in this paper.

\begin{theorem}[\cite{Kirsc1993}, Theorem a]
Suppose the incident field $u^i(\alpha,x)=e^{\i\alpha x_1-\i\sqrt{k^2-\alpha^2}x_2}$ for $\alpha\in(-k,k)$, and it is scattered by a smooth enough sound soft periodic surface.\\ a) The total field, denoted by $u(\alpha,x)$, depends continuously on $\alpha\in(-k,k)$ and analytically on $\alpha$ in the set
\begin{equation*}
\left\{\alpha\in(-k,k):\, (n\Lambda^*-\alpha)^2\neq k^2\text{ for every } n\in\Z\right\}.
\end{equation*}
\\ b) Let $\alpha_0\in(-k,k)$ such that $(n_0\Lambda^*-\alpha_0)^2=k_0^2$ for some $n_0\in\Z$. Then there is a neighbourhood $U$ of $\alpha_0$ and quasi-periodic functions $v(\alpha,\cdot),\,w(\alpha,\cdot)\in\widetilde{H}^1_\alpha(\Omega^\Lambda_H)$ that depend analytically on $\alpha\in U$ and satisfies $u(\alpha,x)=v(\alpha,x)+\sqrt{k^2-|n_0\Lambda^*-\alpha|^2}\, w(\alpha,x)$.
\end{theorem}

From the proof of Theorem a in \cite{Kirsc1993}, the singularity of the total fields with respect to incident plane waves comes from the DtN map. The singularity occurs at the points where at least one term of the DtN map vanishes (so called Wood anomalies, see \cite{Maystre1984}). However, for the scattering problems in this paper, the singularities may also come from the slow decay of the incident fields. In this section, we will only focus on the singularity comes from the DtN map. We will firstly consider the case that the incident field with compact support, and then extend the results to more generalized and commonly used incident fields. We will give a brief description of these singularity results, for details see \cite{Zhang2017d}.

\subsection{Notations and  function spaces}

 Define the set of Wood anomalies by
\begin{equation*}
\S:=\left\{\alpha\in\R:\,\exists\, n\in\Z,\,\text{ such that }\,|\Lambda^*n-\alpha|=k\right\},
\end{equation*}
which is a countable set in $\R$. Let $\left\{\alpha_j\right\}_{j\in\Z}$ be an ascending sequence of all the points in $\S$. To define the continuous or analytical "dependence" rigorously, we have to define the following function spaces first.

Let $\I\subset\R$ be an interval ($\I=\R$ included), $W\subset\R^2$ be a bounded domain, and define the space of functions defined in the  $\I\times W$ that depend analytically on the first variable by
\begin{equation*}
\begin{aligned}
C^\omega(\I;S(W)):=\Bigg\{f\in \D'(\I\times W):\,\forall\,\alpha_0\in\I,\,\exists\,\delta>0,\,s.t.,\,\forall\,\alpha\in(\alpha_0-\delta,\alpha_0+\delta)\cap\I,\Bigg.\\
\left.\exists\, C>0,\,f_n\in S(W),\,s.t.,\,f(\alpha,x)=\sum_{n=0}^\infty (\alpha-\alpha_0)^n f_n(x),\,\left\|f_n\right\|_{S(W)}\leq C^n\right\},
\end{aligned}
\end{equation*}
where $S(W)$ is a Sobolev space defined on the domain $W$. In this paper, $W=\Omega^\Lambda_H$, the function space $S(W)$ could be either $H^n(\Omega^\Lambda_H)$ or $\widetilde{H}^n(\Omega^\Lambda_H)$ for $n=0,1,2$. Define the space of functions that depend $C^n$-continuously on the first variable
\begin{equation*}
\begin{aligned}
C^n(\I;S(W)):=\left\{f\in \D'(\I\times W):\,\forall\alpha\in\I,\,j=0,\dots,n,\frac{\partial^j f(\alpha,\cdot)}{\partial\alpha^j}\in S(W),\right.\\\left.
\text{moreover, }\left\|\frac{\partial^j f(\alpha,\cdot)}{\partial\alpha^j}\right\|_{S(W)}\text{ is uniformly bounded for }\alpha\in\I\right\},
\end{aligned}
\end{equation*}
thus we can define the space $C^\infty(\I;S(W))$ in the same way
\begin{equation*}
\begin{aligned}
C^\infty(\I;S(W)):=\left\{f\in \D'(\I\times W):\,\forall\alpha\in\I,\,j=0,\dots,\infty,\frac{\partial^j f(\alpha,\cdot)}{\partial\alpha^j}\in S(W),\right.\\\left.
\text{moreover, }\left\|\frac{\partial^j f(\alpha,\cdot)}{\partial\alpha^j}\right\|_{S(W)}\text{ is uniformly bounded for }\alpha\in\I\right\},
\end{aligned}
\end{equation*}
We can also define the subspace of  $C^n(\I;S(W))$ for finite interval $\I=[A_0,A_1]$ by
\begin{equation*}
\begin{aligned}
&C_p^n(\I;S(W)):=\Bigg\{f\in C^n(\I; S(W)):\,\forall\,j=0,\dots,n,\Bigg.\\&\qquad\qquad\qquad\qquad\qquad
\left.
\lim_{t\rightarrow A_0^+}\left\|\frac{\partial^j f}{\partial t^j}(t,\cdot)\right\|_{S(W)}=\lim_{t\rightarrow A_1^-}\left\|\frac{\partial^j f}{\partial t^j}(t,\cdot)\right\|_{S(W)}=0\right\}
\end{aligned}
\end{equation*}
Similarly, we can also define $C_p^\infty(\I;S(W))$. Define the subspace $C^\omega(\I;S_\alpha(W))$ by
\begin{equation*}
C^\omega(\I;S_\alpha(W)):=\left\{f(\alpha,\cdot)\in C^\omega(\I;S(W));\,f
\text{ is $\alpha$-quasi-periodic for any  $\alpha\in\I$}\right\},
\end{equation*}
and also define $C^n(\I;S_\alpha(W))$ and $C^\infty(\I;S_\alpha(W))$ in the same way.

\subsection{Regularity properties for Bloch transformed solutions}

Firstly, assume that the incident field $u^i$ is compactly supported in $\Omega^H_p$, thus $u^i\big|_{\Gamma_H}\in H^1_r(\Gamma_H)$  for any  $r\in\R$, especially, for any $r>1/2$. Thus $f(\alpha,\cdot)$ in \eqref{eq:var_bloch_trans_old} belongs to the space $H_0^r(\Wast;H^{-1/2}_\alpha(\Gamma^\Lambda_H))\cap C^\omega(\Wast;H^{-1/2}_\alpha(\Gamma^\Lambda_H))$. From  Theorem \ref{th:continuous}, the Bloch transformed scattered field $w\in H_0^r(\Wast;\widetilde{H}^1_\alpha(\Omega^\Lambda_H))$. Moreover,  for any fixed $\alpha\in\Wast$ and $\phi_\alpha\in\widetilde{H}^1_\alpha(\Omega^\Lambda_H)$,
\begin{equation}
a_\alpha(w(\alpha,\cdot),\phi_\alpha)+b(u_T,\phi_\alpha)=\int_{\Gamma^\Lambda_H} f(\alpha,\cdot)\overline{\phi_\alpha}\d s.
\end{equation}

Set $w_\alpha(x):=\exp(\i\alpha x_1)w(\alpha,x)$, then $w_\alpha$ is a periodic function for any fixed $\alpha\in\W$. Define $\phi_\alpha:=\exp(-\i\alpha x_1)\phi(x)$ for any periodic function $\phi$. Then replace $w$ and $\phi_\alpha$ by $w_\alpha$ and $\phi$, for each $w_\alpha$, it satisfies the following variational equation
\begin{equation}\label{eq:var_periodic}
\int_{\Omega^\Lambda_H}\left[\grad_\alpha w_\alpha\cdot\overline{\grad_\alpha\phi}-k^2w_\alpha\overline{\phi}\right]-\int_{\Gamma^\Lambda_H}\widetilde{T}^+_\alpha\left[w_\alpha\big|_{\Gamma^\Lambda_H}\right]\overline{\phi}\d s=\int_{\Gamma^\Lambda_H}g(\alpha,\cdot)\overline{\phi}\d s-\widetilde{b}_\alpha(u_T,\phi),
\end{equation} 
where $\grad_\alpha=\grad-\i\alpha(1,0)^\top$, $\widetilde{T}^+_\alpha$ is defined by
\begin{equation*}
\widetilde{T}^+_\alpha(\psi)=\i\sum_{j\in\Z}\sqrt{k^2-|\Lambda^*j-\alpha|^2}\widehat{\psi}(j)e^{\i\Lambda^* x_1},\quad\psi=\sum_{j\in\Z}\widehat{\psi}(j)e^{\i\Lambda^* j x_1},
\end{equation*}
the sysquilinear form $\widetilde{b}_\alpha(\cdot,\cdot)$ is defined by
\begin{equation*}
\widetilde{b}_\alpha(\xi,\psi)=\left[\frac{\Lambda}{2\pi}\right]^{1/2}\int_{\Omega^\Lambda_H}e^{\i\alpha x_1}\left[(A_p-I)\grad\xi\cdot\overline{\grad_\alpha\psi}-k^2(c_p-1)\xi\overline{\psi}\right]\d x,
\end{equation*}
and 
\begin{equation*}
g(\alpha,x)=e^{\i\alpha x_1}f(\alpha,x).
\end{equation*}

As $f\in C^\omega(\Wast;H^{-1/2}_\alpha(\Gamma^\Lambda_H))$, the function $g(\alpha,\cdot)\in C^\omega(\Wast;H^{-1/2}_0(\Omega^\Lambda_H))$. So the right hand side of equation \eqref{eq:var_periodic} depends analytically on $\alpha$, for any $g\in C^\omega(\Wast;H^{-1/2}_0(\Omega^\Lambda_H))$ and $\phi\in \widetilde{H}^1_0(\Omega^\Lambda_H)$. From Riesz representation theorem, there is an operator $F_\alpha$ that maps $C^\omega(\Wast;H^{-1/2}_0(\Omega^\Lambda_H))$ to the space $\widetilde{H}^1_0(\Omega^\Lambda_H)$, such that
\begin{equation}
\int_{\Gamma^\Lambda_H}g(\alpha,\cdot)\overline{\phi}\d s-\widetilde{b}_\alpha(u_T,\phi)=\int_{\Omega^\Lambda_H}F_\alpha(g)\overline{\phi}\d x.
\end{equation} 
Moreover, the operator $F_\alpha$ depends analytically on $\alpha$, see Chapter 7, Section 1 in \cite{Kato1995}. Following the proof of Theorem a in \cite{Kirsc1993}, we can obtain the following property of $w_\alpha$ from the singularity with respect to $\alpha$ in the DtN map $\widetilde{T}^+_\alpha$.

\begin{theorem}
Suppose the incident field $u^i\in H^1(\Omega_H)$ has a compact support. Then  $w_\alpha(x)$ that solves \eqref{eq:var_bloch_trans} depends continuously in $\alpha\in\Wast$ and analytically on $\Wast\setminus\S$. If $\alpha_0\in\S$, i.e., there is an $n_0\in\Z$ such that $|n_0\Lambda^*-\alpha_0|=k$. Then there is a neighbourhood $U$ of $\alpha_0$ and two functions $v_1^\alpha,\,v_2^\alpha\in\widetilde{H}^1_0(\Omega^\Lambda_H)$ that depend analytically on $\alpha$, such that $w_\alpha=v_1^\alpha+\sqrt{k^2-|n_0\Lambda^*-\alpha|^2}\,v_2^\alpha$.
\end{theorem}

The singularity result could have a simplified form.

\begin{corollary}\label{th:decomp}
If $\alpha_0\in\S$, then there is a neighbourhood $U$ of $\alpha_0$ and two functions $w_1,\,w_2\in C^\omega(U;\widetilde{H}^1_\alpha(\Omega^\Lambda_H))$  such that $w(\alpha,\cdot)=w_1(\alpha,\cdot)+\sqrt{\alpha-\alpha_0}\,w_2(\alpha,\cdot)$.
\end{corollary}

\begin{proof}
For any $\alpha_0\in\S$, there is either 1) an $n_0\in\Z$ such that $|\Lambda^*n-\alpha_0|=k$ or 2) two $n_1,\,n_2\in\Z$ such that $|\Lambda^*n_1-\alpha_0|=|\Lambda^*n_2-\alpha_0|=k$.

Suppose $|\Lambda^* n_j-\alpha_0|=k$, then $\Lambda^* n_j- k=\alpha_0$ or $\Lambda^* n_j+ k=\alpha_0$. If $\Lambda^* n_j-k=\alpha_0$, 
\begin{equation*}
\sqrt{k^2-|\Lambda^* n_j-\alpha|^2}=\sqrt{\alpha-\alpha_0}\sqrt{2k+\alpha_0-\alpha},
\end{equation*}
where the second term is analytic in a small neighbourhood of $\alpha_0$.  Thus the total field $w_\alpha$ could be written into the form $w_\alpha=w_1^\alpha+\sqrt{\alpha-\alpha_0}\,w_2^\alpha$, where $w_1^\alpha=v_1^\alpha$ and $w_2^\alpha=\sqrt{2k+\alpha_0-\alpha}\,v_2^\alpha$. Let $w_1(\alpha,\cdot)=e^{-\i\alpha x_1}w_1^\alpha$ and $w_2(\alpha,\cdot)=e^{-\i\alpha x_1}w_2^\alpha$, then $w(\alpha,\cdot)=w_1(\alpha,\cdot)+\sqrt{\alpha-\alpha_0}\,w_2(\alpha,\cdot)$. The case that $\Lambda^*n_j+k=\alpha_0$ is similar.  The proof is finished.

\end{proof}

We have just investigated the decomposition of the Bloch transformed field $w(\alpha,\cdot)$, when  $u^i$ is compactly supported. However, the result for this special case could be  extended to more generalized cases. Firstly, define the following two conditions for a function $u\in C^0(\I;S(W))$.
\begin{enumerate}
\item For any subinterval $\I_0\subset\I\setminus\S$, $u\in C^\omega(\I_0;S(W))$.
\item For any $\alpha_j\in\I\cap\S$, there is a small enough $\delta>0$ with two pairs $v_1,\,w_1\in C^\omega((\alpha_j-\delta,\alpha_j];S(W))$ and $v_2,\,w_2\in C^\omega([\alpha_j,\alpha_j+\delta);S(W))$ such that
\begin{eqnarray*}
&&u=v_1+\sqrt{\alpha-\alpha_j}\,w_1\text{ for }\alpha\in(\alpha_j-\delta,\alpha_j],\\
&&u=v_2+\sqrt{\alpha-\alpha_j}\,w_2\text{ for }\alpha\in[\alpha_j,\alpha_j+\delta).
\end{eqnarray*}
\end{enumerate}
Then the function spaces are defined as follows.
\begin{eqnarray*}\A^\omega_c(\I;S(W);\S):=\left\{ u\in C^0(\I;S(W)):\,\text{ $u$ satisfies conditions 1 and 2}\right\},
\end{eqnarray*}
and denote its subspace of all $\alpha$-quasi-periodic functions by $\A^\omega_c(\I;S_\alpha(W);\S)$.

\begin{assumption}\label{asp}
The incident field $u^i\in H_r^2(\Omega^H_p)$ for $r>1/2$ satisfies that its Bloch transform $\J_\Omega u^i\in\A_c^\omega\left(\Wast;H^2_\alpha(\Omega^\Lambda_H);\S\right)$.
\end{assumption}

\begin{remark}
From \cite{Zhang2017d}, the two types of incident fields, i.e., the half space Green's function and the Herglotz wave functions, satisfy Assumption \ref{asp}. Set $\Lambda=2\pi$ for simplicity,  then $\Lambda^*=1$.

The first step is to investigate the function $e^{\i\sqrt{k^2-|n-\alpha|^2} x_2}$. From the Taylor's expansion,
\begin{equation*}
\begin{aligned}
e^{\i\sqrt{k^2-|n-\alpha|^2} x_2}=\sum_{\ell\in\Z}\frac{(\i\sqrt{k^2-|n-\alpha|^2}x_1)^\ell}{\ell!}
\end{aligned}
\end{equation*}
For any $\alpha\in\Wast\setminus\S$, $e^{\i\sqrt{k^2-|n-\alpha|^2} x_2}$ is an analytic function with respect to $\alpha$. For any $\alpha_0\in\S$, there is an integer $n$ such that either $k=n-\alpha_0$ or $-k=n-\alpha_0$. Take  $k=n-\alpha_0$ for example,
\begin{equation*}
\begin{aligned}
&e^{\i\sqrt{k^2-|n-\alpha|^2} x_2}
=\sum_{\ell\in\Z}\frac{(\i\sqrt{k^2-|n-\alpha|^2}x_2)^{2\ell}}{(2\ell)!}+\sum_{\ell\in\Z}\frac{(\i\sqrt{k^2-|n-\alpha|^2}x_2)^{2\ell+1}}{(2\ell+1)!}\\
=&\sum_{\ell\in\Z}\frac{(|n-\alpha|^2-k^2)^{\ell}x_2^{2\ell}}{(2\ell)!}+\i\sqrt{\alpha-\alpha_0}\sqrt{2k+\alpha_0-\alpha}\sum_{\ell\in\Z}\frac{(|n-\alpha|^2-k^2)^{\ell}x_2^{2\ell+1}}{(2\ell+1)!}.
\end{aligned}
\end{equation*}
As the two series above converges absolutely, they are both analytic functions with respect to $\alpha$. In a small neighbourhood of $\alpha_0$, $\sqrt{2k+\alpha_0-\alpha}$ is also analytic, thus there are two functions $\psi_1(\alpha,x_1)$ and $\psi_2(\alpha,x_1)$ that depend analytically in a small neighbourhood of $\alpha_0$ such that
\begin{equation}\label{eq:split}
e^{\i\sqrt{k^2-|n-\alpha|^2} x_1}=\psi_1(\alpha,x_1)+\sqrt{\alpha-\alpha_0}\psi_2(\alpha,x_1).
\end{equation}

\noindent
1) The half space Green's function, defined by
\begin{equation*}
G(x,y)=\frac{\i}{4}\left[H_0^{(1)}(k|x-y|)-H_0^{(1)}(k|x-y'|)\right]
\end{equation*}
with the point $y=(y_1,y_2)^\top$ be above the (locally perturbed) periodic surface and $y'=(y_1,-y_2)^\top$. It is easy to check that the Bloch transform has the following representation
\begin{equation*}
\left(\J_\Omega G(\cdot,y)\right)(\alpha,x)=\frac{1}{2\pi}\sum_{j\in\Z}e^{\i(j-\alpha)(x_1-y_1)+\i\sqrt{k^2-|j-\alpha|^2}y_2}{\rm sinc}\left(\sqrt{k^2-|j-\alpha|^2}x_2\right)x_2,
\end{equation*}
where ${\rm sinc}$ is an analytic function defined by $\sin t/t$ for $t\neq 0$ with ${\rm sinc}(0)=1$. As the Taylor's series ${\rm sinc}=\sum_{j=0}^\infty\frac{(-t^2)^n}{(2n+1)!}$, the function ${\rm sinc}(\sqrt{k^2-|j-\alpha|^2}x_2)$ depends analytically on $\alpha\in\R$. Thus the only factor that may cause singularity is $e^{\i\sqrt{k^2-|j-\alpha|^2}}$ for finite number of $j$'s. From \eqref{eq:split}, the function $\J_\Omega G(\cdot,y)$ also have the similar form of decomposition as \eqref{eq:split}.

\noindent
2) The upward/downward propagating Herglotz wave function defined by
\begin{equation*}
(H^\mp g)(x)=\int_{\S_{\pm}}e^{\i k x\cdot d}g(d)\d s(d)=\int_{-\pi/2}^{\pi/2}e^{\i k(\cos t x_1\pm\sin t x_1)}\phi(t)\d t.
\end{equation*}
where $\S_+$ ($\S_-$) is the upper (lower) unit circle. $(H^\mp g)$ satisfies Assumption \eqref{asp}  when $\phi(t)=h(\cos(t))\cos(t)$ where $h$ is an analytic function defined in $[0,1]$ with $h(0)=0$.

Take $H^-g$ for example. From \cite{Lechl2015e}, the Bloch transform has the following representation
\begin{equation*}
\begin{aligned}
(\J_\Omega H^-g)(\alpha,x)&=\sum_{|j-\alpha|<k}e^{\i(j-\alpha)x_1-\i\sqrt{k^2-|j-\alpha|^2}x_2}\frac{\phi(\arcsin(j-\alpha)/k)}{\sqrt{k^2-|j-\alpha|^2}}\\
&=\frac{1}{k}\sum_{|j-\alpha|<k}e^{\i(j-\alpha)x_1-\i\sqrt{k^2-|j-\alpha|^2}x_2}h\left[\sqrt{1-|j-\alpha|^2/k^2}\right]
\end{aligned}
\end{equation*}
This is a finite series, we only need to investigate it by the decompositions of the exponential terms and the property of the  analytic function $h$. The calculation process is omitted.

\end{remark}

\begin{remark}
The plane waves do not satisfy Assumption \ref{asp}, for the plane waves belong to $H_r^2(\Omega^H_p)$ for $r<-1/2$.
\end{remark}

\high{With the incident fields satisfying Assumption \ref{asp}, the following results hold, see \cite{Zhang2017d}.
}

\begin{theorem}\label{th:reg3}
Suppose $u^i$  satisfies Assumption  \ref{asp}, then $w\in H_0^r\left(\Wast;\widetilde{H}_\alpha^2(\Omega^\Lambda_H)\right)$ satisfies \eqref{eq:var_bloch_trans} belongs to the space $\A_{c}^\omega\left(\Wast;\widetilde{H}^2_\alpha(\Omega^\Lambda_H);\S\right)$.
\end{theorem}

\begin{proof}
When $\alpha_0\in\Wast\setminus\S$, as $\J_\Omega u^i\in \A_c^\omega(\Wast;H_\alpha^2(\Omega^\Lambda_H))$, it depends analytically on $\alpha$ in a neighbourhood $U\subset\Wast\setminus\S$. From Corollary \ref{th:decomp}, $w(\alpha,\cdot)$ also depends analytically on $\alpha\in U$.

When $\alpha_0\in \Wast\cap\S$. Consider the case that $\alpha\in(\alpha_0-\delta,\alpha_0]$. Suppose there are $v_1(\alpha,\cdot),\,v_2(\alpha,\cdot)\in C^\omega((\alpha_0-\delta];H^2_\alpha(\Omega^\Lambda_H))$ such that $\J_\Omega u^i=v_1(\alpha,\cdot)+\sqrt{\alpha-\alpha_0}\,v_2(\alpha,\cdot)$, then the Bloch transformed fields $w_1(\alpha,\cdot),\,w_2(\alpha,\cdot)$ with respect to $v_1(\alpha,\cdot),\,v_2(\alpha,\cdot)$ have the following decompositions
\begin{equation*}
w_1(\alpha,\cdot)=w_1^1(\alpha,\cdot)+\sqrt{\alpha-\alpha_0}\,w_1^2(\alpha,\cdot);\quad w_2^\alpha=w_2^1(\alpha,\cdot)+\sqrt{\alpha-\alpha_0}\,w_2^2(\alpha,\cdot),
\end{equation*}
where $w_j^\ell(\alpha,\cdot)\in C^\omega((\alpha_0-\delta];H^2_\alpha(\Omega^\Lambda_H))$ for $j,\,\ell=1,2$. Then the Bloch transformed field $w$ with respect to the field $\J_\Omega u^i$ could be written into
\begin{equation*}
\begin{aligned}
w(\alpha,\cdot)&=w_1(\alpha,\cdot)+\sqrt{\alpha-\alpha_0}w_2(\alpha,\cdot)\\&=\left[w_1^1(\alpha,\cdot)+(\alpha-\alpha_0)w_2^2(\alpha,\cdot)\right]+\sqrt{\alpha-\alpha_0}\,\left[w_1^2(\alpha,\cdot)+w_2^1(\alpha,\cdot)\right].
\end{aligned}
\end{equation*}
Similar result could also be obtained for $\alpha\in[\alpha_0,\alpha_0+\delta)$. 

Thus $w(\alpha,\cdot)\in\A_c^\omega\left(\Wast;\widetilde{H}^2_\alpha(\Omega^\Lambda_H);\S)\right)$, the proof is finished.
\end{proof}

\high{The decomposition of the Bloch transformed total field could also be written by smooth functions.}

\begin{corollary}\label{th:reg4}
Suppose $u^i$ satisfies Assumption \ref{asp}, then the Bloch transform $w(\alpha,\cdot)=\left(\J_\Omega u\right)(\alpha,\cdot)$ depends continuously in $\alpha\in\Wast$. For each $j\in\Z$, for $\alpha\in[\alpha_j,\alpha_{j+1}]$, there are three  functions $w_\ell^j(\alpha,\cdot)\in C^\infty\left(\left[\alpha_j,\alpha_{j+1}\right];\widetilde{H}^2(\Omega^\Lambda_H)\right)$ where $\ell=0,1,2$ such that
\begin{equation}
w(\alpha,\cdot)=w_0^j(\alpha,\cdot)+\sqrt{\alpha-\alpha_j}\,w_1^j(\alpha,\cdot)+\sqrt{\alpha-\alpha_{j+1}}\,w_2^j(\alpha,\cdot).
\end{equation}
\end{corollary}

\section{The inverse Bloch transform}

In \cite{Lechl2016a,Lechl2016b,Lechl2017}, the accuracy of the numerical inverse Bloch transform plays an important role in the convergence of the numerical schemes. In this section, we will discuss a new numerical method for the inverse Bloch transform $\J^{-1}_\Omega w$. \high{ For references we refer to \cite{Colto2013}, Chap 3.5.}

\high{ From the previous section, the Bloch transformed solution $w(\alpha,\cdot)$ only has square-root like singularities at the points in $\S$, i.e., any $\alpha\in\Wast$ such that $|\Lambda^* n-\alpha|=k$ for some $n\in \N$. Thus there is an $n\in\Z$, such that $\alpha=\Lambda^*n-k$ or $\alpha=\Lambda^*n+k$. Thus we can define the real number $\k$ by 
\begin{equation*}
\k:=\min\left\{|\Lambda^*n-k|:\,n\in\Z\right\},
\end{equation*}
 then $\k\leq \Lambda^*/2$ and $\S$ could also be represented by $\k$
 \begin{equation*}
 \S=\left\{\Lambda^*n\pm\k:\,n\in\Z\right\}.
 \end{equation*} }
 There are two different kinds of representations of points in $\S$ by $\k$.
\begin{itemize}
\item Case 1, $\k=\frac{m\Lambda^*}{2}$ where $m=0,1$, the set of singularities $\S=\{\k+n\Lambda^*:\,n\in\Z\}$.
\item Case 2, $\k\neq\frac{m\Lambda^*}{2}$ for $m=0,1$, $\S=\{\k+m\Lambda^*,\,-\k+n\Lambda^*:\,m,n\in\Z\}$.
\end{itemize}

For simplicity, we redefine one periodic cell $\Wast$ by $\Wast=\left(-\k,\Lambda^*-\k\right]$. Define the singular points in $\overline{\Wast}$ as follows.
\begin{itemize}
\item For Case 1, let $a_0=-\k$, $a_1=\Lambda^*-\k$.
\item For Case 2, let $a_0=-\k$, $a_1=\k$, $a_2=\Lambda^*-\k$.
\end{itemize}
The  redefinition is proper as the field $w(\alpha,\cdot)$ is $\Lambda^*$-periodic in $\alpha$ for any fixed $x$,
\begin{equation*}
\int_{-\Lambda^*/2}^{\Lambda^*/2}w(\alpha,\cdot)\d\alpha=\int_{-\k}^{\Lambda^*-\k}w(\alpha,\cdot)\d\alpha=\J^{-1}_\Omega w.
\end{equation*}

\subsection{Change of variables}\label{sec:def_g}

To design a new algorithm with a greater convergence rate, we intend to use a "better" contour to take place of the straight line in the integral of the inverse Bloch transform. 
The strategy adopted in this paper is illustrated and utilized in Chap 3.5, \cite{Colto2013}. In this section, we will recall the method and apply it to the numerical approximation of the inverse Bloch transform. Define a  monotonic function $g\in C^\infty[a_j,a_{j+1}]$ such that 
\begin{equation*}
g(a_j)=a_{j},\,g(a_{j+1})=a_{j+1};\, g'>0\text{ in }(a_j,a_{j+1}).
\end{equation*}
 Moreover, assume that the following conditions are also satisfied.
\begin{assumption}\label{th:asp}
Assume that there is a $\delta>0$, such that
\begin{eqnarray*}
&&g(t)-a_j=C_j h^2(t-a_j),\quad t\in[a_j,a_j+\delta);\\
&&a_{j+1}-g(t)=C_j h^2(a_{j+1}-t),\quad t\in(a_{j+1}-\delta,a_{j+1}],
\end{eqnarray*}
\high{where the function $h\in C^\infty[0,\delta)$  satisfies either of the following two conditions.}
\begin{enumerate}
\item There is a positive integer $m\in\N$ such that $h(t)=O(t^{m+1})$ as $t\rightarrow 0^+$.
\item $h(t)=o(t^n)$ for any $n\in\N$ as $t\rightarrow 0^+$.
\end{enumerate}
\end{assumption}
The contours  in one periodic cell $\Wast(\Lambda=2\pi)$ are shown in Figure \ref{fig:contour}. For Case 1, we take $k=1$ as an example, and the contour is plotted in (a); for Case 2, $k=1.2$ is taken as an example, see (b).




\begin{figure}[htb]
\centering
\begin{tabular}{c c}
\includegraphics[width=0.4\textwidth]{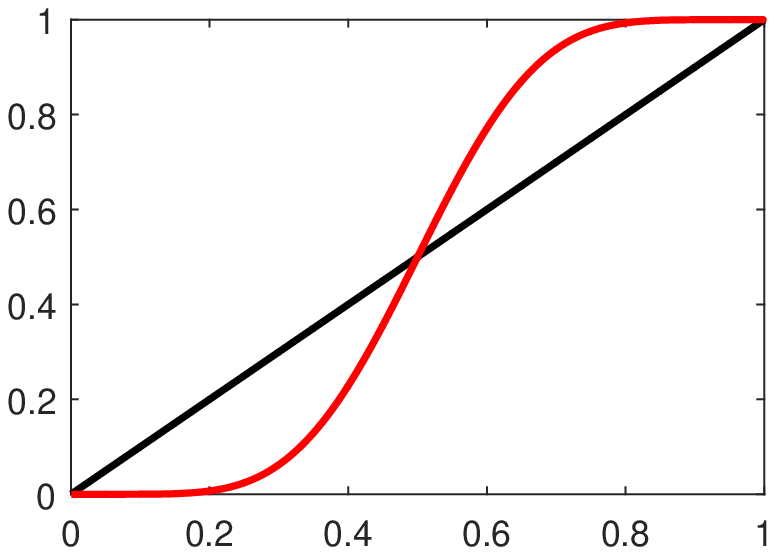} 
& \includegraphics[width=0.4\textwidth]{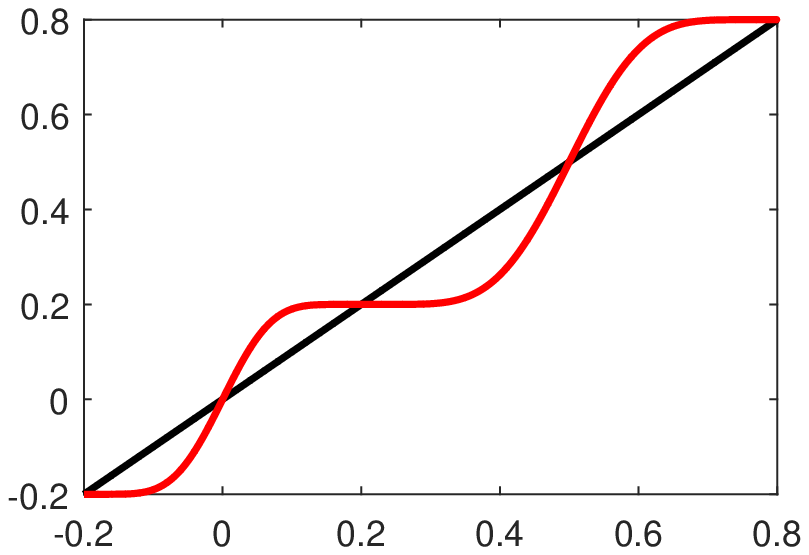}\\[-0cm]
(a) & (b)  
\end{tabular}%
\caption{(a)-(b): The two cases of of the locations of singularities and contours.}
\label{fig:contour}
\end{figure}

With the transform of $\alpha=g(t)$, the inverse Bloch transform has the representation
\begin{equation}\label{eq:inv_Bloch}
\left(\J^{-1}w\right)(x)=\int_{\Wast}w(\alpha,\cdot)\d\alpha=\int_\Wast w(g(t),\cdot)g'(t)\d t.
\end{equation} 
For Case 1, $\Wast=(a_0,a_1]$,
\begin{equation*}
\left(\J^{-1}w\right)(x)=\int_{a_0}^{a_1} w(g(t),\cdot)g'(t)\d t;
\end{equation*} 
for Case 2, $\Wast=(a_0,a_2]$,
\begin{equation*}
\begin{aligned}
\left(\J^{-1}w\right)(x)=\int_{a_0}^{a_1} w(g(t),\cdot)g'(t)\d t+\int_{a_1}^{a_2} w(g(t),\cdot)g'(t)\d t.
\end{aligned}
\end{equation*} 

Let the interval $[A_0,A_1]$ be either $[a_0,a_1]$ (Case 1 \& 2)  or $[a_1,a_2]$ (Case 2). Define 
\begin{equation}\label{eq:integrand}
v(t,\cdot):=w(g(t),\cdot)g'(t),\quad t\in[A_0,A_1],
\end{equation}
then the inverse Bloch transform
\begin{equation}
\left(\J^{-1}_\Omega w\right)(x)=\int_\Wast v(t,x)\d t,\quad x\in\Omega^\Lambda_H.
\end{equation}

We will study the Bloch transform when $w$ satisfies certain conditions.
\begin{assumption}\label{asp2}
$w\in C^0([A_0,A_1];S_\alpha(W))$ has a form of 
\begin{equation*}
w=w_2+\sqrt{\alpha-A_0}\, w_0+\sqrt{\alpha-A_1}\,w_1,
\end{equation*}
 where $w_0,\,w_1,\,w_2\in C^\infty([A_0,A_1];S(W))$.
\end{assumption}

\begin{theorem}\label{th:smoothness}
If there is an $m\in\N$ such that $h(t)=O(t^{m+1})$ as $t\rightarrow 0^+$, for any $w$ satisfies Assumption \ref{asp2}, $v$ defined by \eqref{eq:integrand} belongs to the  function space $C^{2m}_p\left([A_0,\high{A_1}];S_{g(t)}(W)\right)\cap C^\infty\left([A_0,A_1];S_{g(t)}(W)\right)$.
\end{theorem}

\begin{proof}
\high{First, consider the case that $S(W)=L^2(\Omega^\Lambda_H)$.}
Consider the function with $\alpha$ near the point $A_0$. From Assumption \ref{asp2}, for any $0<\delta<A_1-A_0$, in the interval   $[A_0,A_0+\delta]$, $w(\alpha,x)$ has the form of 
\begin{equation*}
w(\alpha,x)=\widetilde{w}_0(\alpha,x)+\sqrt{\alpha-A_0}\,w_0(\alpha,x),
\end{equation*}
where $\widetilde{w}_0(\alpha,x)=w_2(\alpha,x)+\sqrt{\alpha-A_1}\, w_1(\alpha,x)$ belongs to $C^\infty\left([A_0,A_0+\delta];S(W)\right)$. 
Then for any $t\in[A_0,A_0+\delta]$,
\begin{equation*}
\begin{aligned}
v(t,x)&=\widetilde{w}_0(g(t),x)g'(t)+w_0(g(t),x)\sqrt{g(t)-A_0}\, g'(t)\\
&=2 C_0 \widetilde{w}_0(A_0+h^2(t-A_0),x)h(t-A_0)h'(t-A_0)\\
&+2 C_0 w_0(A_0+h^2(t-A_0),x)h^2(t-A_0)h'(t-A_0).
\end{aligned}
\end{equation*}
\high{From the representation of $\widetilde{w}_0$ and $v$, 
\begin{equation*}
\|v(t,\cdot)\|_{S(W)}\leq C(g)\left[\|w_0\|_{S(W)}+\|w_1\|_{S(W)}+\|w_2\|_{S(W)}\right].
\end{equation*}
}
Take the first term for example. As $h\in C^\infty[0,\delta)$ and $g\in C^\infty[A_0,A_1]$, the function $v\in C^\infty([A_0,A_1];S_{g(t)}(W))$.  As $h(t)=O(t^{m+1})$, $h(t-A_0)h'(t-A_0)=O((t-A_0)^{2m+1})$, then 
\begin{equation*}
\lim_{t\rightarrow A_0^+}\left[h(t-A_0)h'(t-A_0)\right]^{(\ell)}(t)=0, \,\quad\text{ for any }\ell=0,1,\dots,2m.
\end{equation*} 
From the $\ell$-th derivative w.r.t. $t$
\begin{equation*}
\begin{aligned}
&\frac{\partial^\ell }{\partial t^\ell}\left[\widetilde{w}_0(A_0+h^2(t-A_0),x)h(t-A_0)h'(t-A_0)\right]\\
=
&\sum_{j=0}^\ell\left(\begin{matrix}
\ell\\j
\end{matrix}\right)\frac{\partial^j}{\partial t^j}\left[\widetilde{w}_0(A_0+h^2(t-A_0),x)\right]\frac{\partial^{\ell-j}}{\partial t^{\ell-j}}\left[h(t-A_0)h'(t-A_0)\right],
\end{aligned}
\end{equation*}
in each term, the first element $\frac{\partial^j}{\partial t^j}\left[\widetilde{w}_0(A_0+h^2(t-A_0),x)\right]$ is a finite sum of the terms if form of $\frac{\partial^\ell}{\partial t^\ell}\left[\widetilde{w}_0(A_0+h^2(t-A_0),x)\right]\Pi_{m=0}^N \left[h^{(m)}(t-A_0)\right]^{n(m)}$, where $N$ is some positive integer and $n(m)\geq 0$ in $\N$. Thus the norm $\left\|\frac{\partial^j}{\partial t^j}\left[\widetilde{w}_0(A_0+h^2(t-A_0),x)\right]\right\|_{S(W)}$ is uniformly bounded for $t\rightarrow A_0^+$. From $\lim_{t\rightarrow A_0^+}\left[h(t-A_0)h'(t-A_0)\right]^{\ell}(t)=0$, 
\begin{equation*}
\left\|\frac{\partial^\ell v(t,x)}{\partial t^\ell}\right\|_{S(W)}\rightarrow 0\quad \text{as }t\rightarrow A_0^+.
\end{equation*}
Similarly, we can also prove that $\left\|\frac{\partial^\ell v(t,x)}{\partial t^\ell}\right\|_{S(W)}\rightarrow 0$ as $t\rightarrow A_1^-$, $\ell=0,1,\dots,2m$, thus $v\in C_p^{2m}\left([A_0,A_1];S_{g(t)}(W)\right)$. 

\high{The case that $S(W)=H^n(\Omega^\Lambda_H)$ for any $n\in\N$ could be proved similarly, thus we omit the proof here.}
\end{proof}

A direct corollary shows a similar result for the case that $h(t)=o(t^n)$ as $t\rightarrow 0^+$, $\forall n$.
\begin{corollary}
If $h(t)=o(t^n)$ as $t\rightarrow 0^+$ for any positive integer $n$, for any $w$ satisfies Assumption \ref{asp2}, $v\in C_p^\infty\left([A_0,A_1];S_{g(t)}(W)\right)$.
\end{corollary} 

From the smoothness of the function $v(t,\cdot)$, the bounds of the derivatives of $v$ depend on the functions $w$ and $g$, and the order of the derivative.

\begin{theorem}\label{th:regul_trans}
Suppose $w$ satisfies Assumption \ref{asp2}, $v$ is defined by \eqref{eq:integrand}. If $h(t)=O(t^{m+1})$ for some $m\in\N$, then $v$ belongs to the functions space $C_p^{2m}([A_0,A_1];S_{g(t)}(W))\cap C^\infty([A_0,A_1];S_{g(t)}(W))$ with
\high{\begin{equation*}
\left\|\frac{\partial^j}{\partial t^j}v(t,\cdot)\right\|_{S(W)}\leq C(j,g)\vertiii{w}_{S(W)},\quad j=0,1,\dots,2m;
\end{equation*}}
if $h(t)=o(t^n)$ for any $n\in\N$, then $v\in C_p^\infty([A_0,A_1];S_{g(t)}(W))$ with 
\high{\begin{equation*}
\left\|\frac{\partial^j}{\partial t^j}v(t,\cdot)\right\|_{S(W)}\leq C(j,g)\vertiii{w}_{S(W)},\quad j=0,1,\dots,\infty,
\end{equation*}
where $C(j,g)$ is a constant \high{that} depends on $j$ and $g$, where 
\begin{equation*}
\vertiii{w}_{S(W)}:=\|w_0\|_{S(W)}+\|w_1\|_{S(W)}+\|w_2\|_{S(W)}.
\end{equation*}
}
\end{theorem}

\subsection{Numerical method for the inverse Bloch transform}

In this subsection, we will analyse the new numerical method of the inverse Bloch transform. Consider the definite integral
\begin{equation*}
\int_{A_0}^{A_1} v(t,\cdot)\d t=\int_{A_0}^{A_1}w(\alpha,\cdot)\d\alpha,
\end{equation*} 
where $w$ satisfies Assumption \ref{asp2} and $v$ is defined by \eqref{eq:integrand}. \high{We will approximate the function $v$ by trigonometrical interpolation (for references see \cite{Atkin1989}), and then study the convergence of the numerical integration based on the approximation.} Let $[A_0,A_1]$ be divided uniformly into $N$ subintervals, where $N$ is assumed to be even in this paper. Let the uniformly located grid points
\begin{equation}\label{eq:gridpoitns}
t_j=A_0+\frac{A_1-A_0}{N}j, \quad j=1,\dots,N, 
\end{equation}
and define the basic functions by
\begin{equation}\label{eq:basis}
\phi_N^{(m)}(t)=\frac{1}{N}\sum_{l=-N/2+1}^{N/2}\exp\left(\i l (t-t_m)\frac{2\pi}{A_1-A_0} \right),\,j=1,2,\dots,N,
\end{equation}
then $\phi_N^{(m)}(t_\ell)=\delta_{\ell,m}$, where
\begin{equation*}
\delta_{\ell,m}=\begin{cases}
1,\quad\ell=m;\\
0,\quad\ell\neq m.
\end{cases}
\end{equation*} 

\begin{lemma}\label{th:basis_fourier}
The basic functions $\left\{\phi_N^{(m)}\right\}_{m=1}^N$  are orthogonal.
\end{lemma}

\begin{proof}
For any $j,j'\in\left\{1,\dots,N\right\}$, the inner product shows
\begin{equation*}
\begin{aligned}
&\int_{A_0}^{A_1} \phi^{(j)}_N(t)\overline{\phi^{(j')}_N}(t)\d t\\
=&\frac{1}{N^2}\int_{A_0}^{A_1} \sum_{\ell=-N/2+1}^{N/2}\sum_{\ell'=-N/2+1}^{N/2}\exp\left(\i \ell (t-t_j)\frac{2\pi}{A_1-A_0} \right)\exp\left(-\i \ell' (t-t_{j'})\frac{2\pi}{A_1-A_0} \right)\d t\\
=&\frac{A_1-A_0}{N^2}\sum_{\ell=-N/2+1}^{N/2}\exp\left(\i\ell(t_{j'}-t_j)\frac{2\pi}{A_1-A_0}\right)=\frac{A_1-A_0}{N^2}\sum_{\ell=-N/2+1}^{N/2}\exp\left(\i\ell (j'-j)\frac{2\pi}{N}\right)
\end{aligned},
\end{equation*} 
then the integral equals to 0 if and only if $j\neq j'$. Thus
\begin{equation*}
\int_{A_0}^{A_1} \phi^{(j)}_N(t)\overline{\phi^{(j')}_N}(t)\d t=\frac{A_1-A_0}{N}\delta_{j,j'}.
\end{equation*}
The basic functions $\left\{\phi_N^{(m)}\right\}_{m=1}^N$  are orthogonal. The proof is finished.
\end{proof}

Let $v_N$ be the interpolation of $v$ in the space ${\rm span}\left\{\phi_N^{(m)}\right\}_{m=1}^N$, with the representation
\begin{equation}\label{eq:interpolation}
v_N(t,\cdot)=\sum_{\ell=1}^N \phi^{(\ell)}_N(t)v(t_\ell,\cdot),
\end{equation}
then  $v_N(t_\ell,\cdot)=v(t_\ell,\cdot)\in S(W)$ for $\ell=1,2,\dots,N$. Then the definite integration 
\begin{equation*}
\int_{A_0}^{A_1}v_N(t,\cdot)\d t=\sum_{\ell=1}^N v(t_\ell,\cdot)\int_{A_0}^{A_1}\phi^{(\ell)}_N(t)\d t=\frac{{A_1}-{A_0}}{N}\sum_{j=1}^N v(t_j,\cdot),
\end{equation*}
which coincides with the formula of the trapezoidal rule. 

We will study the error estimate of the trapezoidal rule of the integration of $v$. Let's begin with the Fourier series of the function $v(t,\cdot)$.  A classical Minkowski integral inequality is needed, see Theorem 202, \cite{Hardy1988}.

\begin{lemma}\label{th:minkowski}
Suppose $(S_1,\mu_1)$ and $(S_2,\mu_2)$ are two measure spaces and $F:\,S_1\times S_2\rightarrow \R$ is measurable. Then the following inequality holds for any $p\geq 1$
\begin{equation*}
\left[\int_{S_2}\left|\int_{S_1}F(y,z)\d \mu_1(y)\right|^p\d \mu_2(z)\right]^{1/p}\leq \int_{S_1}\left(\int_{S_2}\left|F(y,z)\right|^p\d \mu_2(z)\right)^{1/p}\d\mu_1(y)
\end{equation*}
\end{lemma}

\begin{lemma}\label{th:Fourier}
\high{Suppose $v$ belongs to $\in C_p^{2m}([A_0,A_1];S_{g(t)}(W))\cap C^\infty([A_0,A_1];S_{g(t)}(W))$}. If $h=O(t^{m+1})$ as $t\rightarrow 0^+$ for some positive integer $m$, then $v$ has the Fourier series with respect to $t$ has the form
 \begin{equation*}
v(t,x)=\sum_{j\in\N}\widehat{v}_j(x)e^{\i j t\frac{2\pi}{A_1-A_0}},
\end{equation*}
where $\widehat{v}_j\in S(W)$ and its norm for $j\in\N^+$
\begin{equation*}
\|\widehat{v}_j\|_{S(W)}\leq\high{\left|\frac{A_1-A_0}{2\pi j}\right|^{2m+1}\|v\|_{C^{2m+1}([A_0,A_1];S(W))}}.
\end{equation*}
\end{lemma}

\begin{proof}\high{First, consider the case that $S(W)=L^2(\Omega^\Lambda_H)$.}
As $v(t,\cdot)\in C_p^{2m}([A_0,A_1];S_{g(t)}(W))\cap C^\infty([A_0,A_1];S_{g(t)}(W))$,  for any $j=0,\dots,2m$, $\frac{\partial^j}{\partial t^j}v(t,x)\in L^2([A_0,A_1];S(W))$. So the Fourier series of $v(t,x)$ with respect to $t$ exists.  From Theorem \ref{th:smoothness}, the function $v\in C_p^{2m}\left([A_0,A_1];S(W)\right)$. By integration by part, for $j\neq 0$
\begin{equation*}
\begin{aligned}
\widehat{v}_j(x)&=\frac{1}{A_1-A_0}\int_{A_0}^{A_1} v(t,x) e^{-\i j t\frac{2\pi}{A_1-A_0}}\d t\\
&=\frac{1}{A_1-A_0}\left(\frac{A_1-A_0}{2\i\pi j }\right)^{2m+1}\int_{A_0}^{A_1}\frac{d^{2m+1}}{d t^{2m+1}}v(t,x)e^{-\i j t\frac{2\pi}{A_1-A_0}}\d t,
\end{aligned}
\end{equation*} 
then from Lemma \ref{th:minkowski},
\begin{equation*}
\begin{aligned}
&\left\|\int_{A_0}^{A_1}\frac{d^{2m+1}}{d t^{2m+1}}v(t,x)e^{-\i j t\frac{2\pi}{A_1-A_0}}\d t\right\|_{S(W)}\\
=&\left(\int_{\Omega^\Lambda_H}\left|\int_{A_0}^{A_1}\frac{d^{2m+1}}{d t^{2m+1}}v(t,x)e^{-\i j t\frac{2\pi}{A_1-A_0}}\d t\right|^2\d x\right)^{1/2}\\
\leq& \int_{A_0}^{A_1}\left(\int_{\Omega^\Lambda_H} \left|\frac{d^{2m+1}}{d t^{2m+1}}v(t,x)e^{-\i j t\frac{2\pi}{A_1-A_0}}\right|^2\d x\right)^{1/2}\d t\\
\leq& \int_{A_0}^{A_1}\left\|\frac{d^{2m+1}}{dt^{2m+1}}v(t,\cdot)\right\|_{S(W)}\d t\leq (A_1-A_0)\high{\|v\|_{C^{2m+1}([A_0,A_1];S(W))}},
\end{aligned}
\end{equation*}
then
\begin{equation*}
\left\|\widehat{v}_j(x)\right\|_{S(W)}\leq \left|\frac{A_1-A_0}{2\pi j}\right|^{2m+1}\high{\|v\|_{C^{2m+1}([A_0,A_1];S(W))}}.
\end{equation*}

\high{The case that $S(W)=H^n(\Omega^\Lambda_H)$ for any $n\in\N$ could be proved similarly, thus we omit the proof here.} The proof is finished.
\end{proof}

A direct corollary of the lemma shows the case that $h(t)=o(t^n)$ for any $n\in\N$. 
\begin{corollary}\label{th:Fourier1}
If $h=o(t^n)$ as $t\rightarrow 0^+$ for any positive integer $n$, then $v$ has the Fourier series with respect to $t$ has the form
 \begin{equation}
v(t,x)=\sum_{j\in\N}\widehat{v}_j(x)e^{\i j t\frac{2\pi}{A_1-A_0}},
\end{equation}
where $\widehat{v}_j\in S(W)$ and its norm
\begin{equation}
\|\widehat{v}_j\|_{S(W)}=o(j^{-2n}).
\end{equation}
\end{corollary}

Then we can get the error between the exact function $v$ and its Fourier interpolation $v_N$ defined in \eqref{eq:interpolation} for $\alpha\in\Wast$, when $N$ is large enough.

\begin{theorem}\label{th:interpolation}
If $v\in \in C_p^{2n}([A_0,A_1];S_{g(t)}(W))\cap C^\infty([A_0,A_1];S_{g(t)}(W))$ for some $n\in\N$, then the  difference between $v(t,\cdot)$ and its approximation $v_N(t,\cdot)$ is bounded by
\begin{equation}
\|v-v_N\|_{L^2([A_0,A_1];S(W))}\leq\high{CN^{-2n-1/2}\|v\|_{C_p^{2m}([A_0,A_1];S(W))}}.
\end{equation}
\end{theorem}

\begin{proof}\high{First, let $S(W)=L^2(\Omega^\Lambda_H)$.}
From the definition of $v_N(t,\cdot)$, \high{i.e., the projection of $v$ into the subspace ${\rm span}\left\{c_n(x)\exp\left[\i n\frac{2\pi}{A_1-A_0}\right],\,n=-N/2+1,\cdots,N/2,\,c_n\in S(W)\right\}$},
\begin{equation*}
\begin{aligned}
\left\| v- v_N\right\|^2_{L^2([A_0,A_1];S(W))}
=& \int_{A_0}^{A_1}\int_{\Omega^\Lambda_H}\left|v-v_N\right|^2\d x\d t\\
=&\int_{A_0}^{A_1}\int_{\Omega^\Lambda_H}\left|\sum_{j\in\Z\setminus[-N/2+1,N/2]} \widehat{v}_l e^{\i j t\frac{2\pi}{A_1-A_0}}\right|^2\d x\d t\\
\leq & (A_1-A_0)\sum_{j\in\Z\setminus[-N/2+1,N/2]}\left\| \widehat{v}_j\right\|_{S(W)}^2.
\end{aligned}
\end{equation*}
With the result of Lemma \ref{th:Fourier},
\begin{equation*}
\begin{aligned}
\left\|v-v_N\right\|^2_{L^2([A_0,A_1],S(W))}
&\leq (A_1-A_0)\sum_{j\in\Z\setminus[-N/2+1,N/2]}\left|\frac{A_1-A_0}{2\pi j}\right|^{4n+2}\|v\|^2_{C^{2n+1}([A_0,A_1];S(W))}\\
&\leq CN^{-4n-1}\|v\|^2_{C_p^{2m}([A_0,A_1];S(W))}.
\end{aligned}
\end{equation*} 

\high{The case that $S(W)=H^n(\Omega^\Lambda_H)$ for any $n\in\N$ could be proved similarly, thus we omit the proof here.}
 The proof is finished.

\end{proof}

The error estimation of the trapezoidal rule will be shown in the following theorem.

\begin{theorem}\label{th:invBloch_num} If $v\in  C_p^{2n}([A_0,A_1];S_{g(t)}(W))\cap C^{\infty}([A_0,A_1];S_{g(t)}(W))$ for some $n\in\N$, then the error of the numerical integration of $S$ is bounded by
\begin{equation}
\left\|\int_{A_0}^{A_1}(v-v_N)(t,\cdot)\d t\right\|_{S(W)}\leq\high{CN^{-2n-1/2}\|v\|_{C^{2n+1}([A_0,A_1];S(W))}}.
\end{equation}
\end{theorem}

\begin{proof}\high{First, let $S(W)=L^2(\Omega^\Lambda_H)$.}
With the result of Theorem \ref{th:interpolation} and Theorem \ref{th:minkowski},
\begin{equation*}
\begin{aligned}
\left\|\int_{A_0}^{A_1}(v- v_N)(t,\cdot)\d t\right\|^2_{S(W)}
=&\int_{\Omega^\Lambda_H}\left|\int_{A_0}^{A_1}(v-v_N)(t,\cdot)\d t\right|^2\d x\\
\leq& \int_{A_0}^{A_1}\left\|v-v_N\right\|^2_{L^2(\Omega^\Lambda_H)}\d t\\
\leq& \|v-v_N\|^2_{L^2([A_0,A_1];S(W))}.
\end{aligned}
\end{equation*}
Then 
\begin{equation*}
\begin{aligned}
\left\|\int_{A_0}^{A_1}\left(v-v_N\right)(t,\cdot)\d t\right\|_{S(W)}&\leq \|v-v_N\|_{L^2([A_0,A_1],S(W))}\\
&\leq\high{CN^{-2n-1/2}\|v\|_{C^{2n+1}([A_0,A_1];S(W))}}.
\end{aligned}
\end{equation*}

\high{The case that $S(W)=H^n(\Omega^\Lambda_H)$ for any $n\in\N$ could be proved similarly, thus we omit the proof here.}
then the proof is finished.
\end{proof}

The result could be applied to the error estimation of  the numerical  inverse Bloch transform \eqref{eq:inv_Bloch}.  
\begin{theorem}
Given a positive integer $N$,   define the nodal points and inverse Bloch trasnform for the two conditions.
\begin{enumerate}
\item If $\Wast=[a_0,a_1]$, then let $t_j=a_0+j(a_1-a_0)/N$, then the inverse Bloch transform
\begin{equation}
(\J^{-1}_{\Omega.N}w)(x)=\frac{a_1-a_0}{N}\sum_{j=1}^N v(g(t_j),x)g'(t_j)
\end{equation}
\item if $\Wast=[a_0,a_2]$, then let $t_j=a_0+j(a_1-a_0)/N$, $t_{j+N}=a_1+j(a_2-a_1)/N$, then the inverse Bloch transform
\begin{equation}
(\J^{-1}_{\Omega.N}w)(x)=\frac{a_1-a_0}{N}\sum_{j=1}^N v(g(t_j),x)g'(t_j)+\frac{a_2-a_1}{N}\sum_{j=N+1}^{2N} v(g(t_j),x)g'(t_j).
\end{equation}
\end{enumerate}
If $h(t)=O(h^{m+1})$ for some $m\in\N$,  the error of the numerical inverse Bloch transform is bounded by
\begin{equation*}
\|\J^{-1}_{\Omega,N}w-\J^{-1}_\Omega w\|_{S(W)}\leq\high{CN^{-2m-1/2}\|v\|_{C^{2m+1}([A_0,A_1];S(W))}},\quad\ell=1,2;
\end{equation*}
if $h(t)=o(h^{n+1})$ for any $n\in\N$,   the error of the numerical inverse Bloch transform is bounded by
\begin{equation*}
\|\J^{-1}_{\Omega,N}w-\J^{-1}_\Omega w\|_{S(W)}\leq\high{CN^{-2n-1/2}\|v\|_{C^{2n+1}([A_0,A_1];S(W))}},\quad\ell=1,2,
\end{equation*}
for any $n\in\N$.
\end{theorem}

\section{Finite element method for the scattering problems}

In this section, we assume that the incident field $u^i$ satisfies Assumption \ref{asp}. With the transform that $\alpha=g(t)$ in $\Wast$, the transformed inverse Bloch transform 
\begin{equation*}
(\widetilde{\J}^{-1}_\Omega v)\left(\begin{matrix}
x_1+\Lambda j\\
x_2
\end{matrix}\right)=\left[\frac{\Lambda}{2\pi}\right]^{1/2}\int_\Wast v(t,x)e^{-\i g(t)\Lambda j}\d t\quad x\in\Omega^\Lambda_H.
\end{equation*}
\high{Replace $\alpha$ by $g(t)$ in \eqref{eq:var_bloch}, then the function $v(t,x)$ satisfies the following equation for any $\phi\in L^2(\Wast;H^1_{g(t)}(\Omega^\Lambda_H))$}
\begin{equation}
\label{eq:var_bloch_trans_old}
\int_{\Wast}a_{g(t)}(v(t,\cdot),\phi(t,\cdot))\d t+ b\left(\widetilde{\J}^{-1}_\Omega v, \widetilde{\J}^{-1}_\Omega\left[\phi(t,\cdot)g'(t)\right]\right)=\int_\Wast\int_{\Gamma^\Lambda_H}F(t,\cdot)\overline{\phi}(t,\cdot)\d t
\end{equation}
where $F(t,\cdot)=f(g(t),\cdot)g'(t)$. 

\high{Define the  space $L^2(\Wast;H^s_{g(t)}(\Omega^\Lambda_H);g)$ by 
\begin{equation*}
L^2(\Wast;H^s_{g(t)}(\Omega^\Lambda_H);g):=\left\{\phi(t,\cdot)=\psi(g(t),\cdot)g'(t);\,\psi(\alpha,\cdot)\in L^2(\Wast;H^s_\alpha(\Omega^\Lambda_H))\right\}.
\end{equation*}
Then define the norm in $L^2(\Wast;H^s_{g(t)}(\Omega^\Lambda_H);g)$ by
\begin{equation*}
\|\phi\|^2_{L^2(\Wast;H^s_{g(t)}(\Omega^\Lambda_H);g)}=\|\psi\|^2_{L^2(\Wast;H^s_\alpha(\Omega^\Lambda_H))}=\int_\Wast \|\psi(g(t),\cdot)\|^2_{H^s_{g(t)}(\Omega^\Lambda_H)}g'(t)\d t.
\end{equation*}
We can also define the function space $H_0^r(\Wast;H^s_{g(t)}(\Omega^\Lambda_H);g)$ in the same way. 
By replacing $\phi(t,\cdot)g'(t)$ with $\psi(t,\cdot)\in L^2(\Wast;H^s_{g(t)}(\Omega^\Lambda_H);g)$, the variational problem \eqref{eq:var_bloch_trans_old} is equivalent to
\begin{equation}\label{eq:var_bloch_trans}
\begin{aligned}
\int_{\Wast}a_{g(t)}(v(t,\cdot),\psi(t,\cdot))\left[g'(t)\right]^{-1}\d t+ b\left(\widetilde{\J}^{-1}_\Omega v, \widetilde{\J}^{-1}_\Omega\psi\right)=\int_\Wast\int_{\Gamma^\Lambda_H}F(t,\cdot)\overline{\psi}(t,\cdot)\left[g'(t)\right]^{-1}\d t.
\end{aligned}
\end{equation} 
The transformed inverse Bloch transform $\widetilde{J}^{-1}_\Omega$ is also defined in $H_0^r(\Wast;H^s_{g(t)}(\Omega^\Lambda_H);g)$, we will show that the operator is bounded. Take $v\in L^2(\Wast;L^2(\Omega^\Lambda_H);g)$,
\begin{equation*}
\left\|\widetilde{\J}_\Omega^{-1}v\right\|^2_{L^2(\Omega^H)}=
\left\|\J_\Omega^{-1} w\right\|^2_{L^2(\Omega^H)}\leq \|w\|^2_{L^2(\Wast;L^2(\Omega^\Lambda_H))}=\|v\|^2_{L^2(\Wast;L^2(\Omega^\Lambda_H);g)}.
\end{equation*}
We can also prove that $\widetilde{\J}^{-1}_\Omega$ is a bounded operator from $H_0^r(\Wast;H^s_{g(t)}(\Omega^\Lambda_H);g)$ to $H_r^s(\Omega^H)$.

From the procedure in the arguments above, the following equivalence result is obtained by simply change of variables.
\begin{lemma}\label{lm:ini_solv}
Suppose the incident field $u^i\in H_r^1(\Omega^H_p)$ for some $r\in[0,1)$ and the functions $\zeta,\,\zeta_p$ are Lipschitz continuous functions. Then $w=\J_\Omega u_T\in H_0^r(\Wast;\widetilde{H}^1_\alpha(\Omega_H^\Lambda))$ satisfies  variational problem \eqref{eq:var_bloch} if and only if $v\in H_0^r(\Wast;\widetilde{H}^1_\alpha(\Omega_H^\Lambda);g)$ is the solution to \eqref{eq:var_bloch_trans}.
\end{lemma}

From the equivalence between \eqref{eq:var_bloch_trans} and \eqref{eq:var_bloch}, the following result is a simple corollary of Theorem \ref{th:uni_solv} and Theorem \ref{th:higherregu}.

\begin{theorem}The variational problem \eqref{eq:var_bloch_trans} is uniquely solvable in $H_0^r(\Wast;\widetilde{H}^1_\alpha(\Omega_H^\Lambda);g)$.  Further more, if $u^i$ satisfies Assumption \ref{asp},  $\zeta,\,\zeta_p\in C^{2,1}(\R)$, when $h(t)=O(t^{m+1})$, the unique solution $v$ to \eqref{eq:var_bloch_trans} belongs to the space $C_p^{2n}(\Wast;\widetilde{H}^2_{g(t)}(\Omega^\Lambda_H))$. 
\end{theorem}

\begin{proof}
We only need to prove the unique solvability. Suppose $v\in H_0^r(\Wast;H^s_{g(t)}(\Omega^\Lambda_H);g)$ is the solution to \eqref{eq:var_bloch_trans} for $F=0$, from the definition of the space $H_0^r(\Wast;H^s_{g(t)}(\Omega^\Lambda_H);g)$, there is a $w\in H_0^r(\Wast;\widetilde{H}^1_\alpha(\Omega^\Lambda_H))$ such that $v(t,\cdot)=w(g(t),\cdot)g'(t)$. From Lemma \ref{lm:ini_solv}, $w(\alpha,\cdot)$ is the unique solution to \eqref{eq:var_bloch} with $f=0$. From Theorem \ref{th:uni_solv}, the variational problem \eqref{eq:var_bloch} has at most one solution, so $w=0$, then $v=0$. Then the injectivity is proved. From the Bounded Inverse Theorem, the problem is uniquely solvable. 

When $u^i$ satisfies Assumption \ref{asp} and $\zeta,\,\zeta_p\in C^{2,1}(\R)$, the solution $w$ to \eqref{eq:var_bloch} belongs to the space $H_0^r(\Wast;\widetilde{H}^2_\alpha(\Omega^\Lambda_H))$ and belongs to $\mathcal{A}_c^\omega\left(\Wast;\widetilde{H}^2_\alpha(\Omega^\Lambda_H);\S\right)$, then from Corollary \ref{th:reg4} $w$ satisfies Assumption \ref{asp2}. From Theorem \ref{th:regul_trans}, when $h=O(t^{m+1})$, $v$ belongs to the space $C_p^{2n}([a_0,a_1];\widetilde{H}^2_{g(t)}(\Omega^\Lambda_H))$ for any $a_0,\,a_1\in\S$ such that $(a_0,a_1)\cap\S$. From the definition of the space $C_p^{2n}(\Wast;\widetilde{H}^2_{g(t)}(\Omega^\Lambda_H))$ and the new definition of $\Wast$,  $v\in C_p^{2n}(\Wast;\widetilde{H}^2_{g(t)}(\Omega^\Lambda_H))\subset H_0^{2n}(\Wast;\widetilde{H}^2_{g(t)}(\Omega^\Lambda_H))$. 
\end{proof}
}

We will build up the finite element space, take $\Wast=[a_0,a_1]$ for example. Let the uniformly located grid points be $\left\{t_j\right\}_{j=1}^N$ and Fourier basic functions $\phi^{(j)}_N$ be defined as in \eqref{eq:gridpoitns}-\eqref{eq:basis}. For notations defined in $\Omega^\Lambda_H$, recall the triangular mesh  $\mathcal{M}_H$, the piecewise linear basic functions  $\left\{\psi^{(\ell)}_M\right\}_{\ell=1}^M$ and the finite dimensional subspace $V_0={\rm span}\left\{\phi^{(\ell)}_M\right\}_{\ell=1}^M\subset H_0^1(\Omega^\Lambda_H)$. Define the finite element space $\widetilde{X}_{N,h}$ by
\begin{equation*}
\widetilde{X}_{N,h}=\left\{v_{N,h}(t,x)=\sum_{j=1}^N\sum_{l=1}^M v^{(j,\ell)}_{N,h}\exp\left(-\i g(t_j)x_1\right)\phi^{j}_N(t)\psi^{\ell}_M(x)\right\}.
\end{equation*}

\begin{theorem}\label{th:interp1}
If $v\in  C_p^{2n}([a_0,a_1];H^2_{g(t)}(\Omega^\Lambda_H))$, then the error between the approximation of $v$ in $\widetilde{X}_{N,h}$, i.e., $v_N$, and $v$, is bounded by
\begin{equation*}
\min_{v_{N,h}\in \widetilde{X}_{N,h}}\|v-v_{N,h}\|_{L^2([A_0,A_1];H^1_{g(t)}(\Omega^\Lambda_H))}\high{\leq}C\left(N^{-2n-1/2}+h\right)\high{\|v\|_{C^{2n+1}([A_0,A_1];H^2_{g(t)}(\Omega^\Lambda_H))}},
\end{equation*}
for large enough $N$ and small enough $h>0$, where $C$ depends on $n$ and $g$.
\end{theorem}

\begin{proof}
From Theorem \ref{th:interpolation}, $v_N(t,x)=\sum_{j=1}^N\phi_N^{(j)}(t)v(t_j,x)$ is an approximation of $v$ with the error 
\begin{equation}\label{eq:err1}
\|v-v_N\|_{L^2(\Wast;H_{g(t)}^1(\Omega^\Lambda_H))}\leq\high{CN^{-2m-1/2}\|v\|_{C^{2m+1}([A_0,A_1];H_{g(t)}^1(\Omega^\Lambda_H))}}.
\end{equation} 

For each $t_j$, the function $\exp\left(\i g(t_j)x_1\right)v(t_j,x)$ is approximated by functions in the subspace $V_0={\rm span}\left\{\psi^{(\ell)}_M,\,\ell=1,\dots,M\right\}$, i.e., there is a series of coefficients $v^{(j,\ell)}_{N,h}\in \C$ such that
\begin{equation*}
v(t_j,x)=\sum_{\ell=1}^M v^{(j,\ell)}_{N,h}\exp\left(-\i g(t_j)x_1\right)\psi^{\ell}_M(x)
\end{equation*}
for any $x$ in the grid points of $\mathcal{M}_h$, with the error estimate
\begin{equation*}
\begin{aligned}
\left\|v(t_j,\cdot)-\sum_{\ell=1}^M v^{(j,\ell)}_{N,h}\exp\left(-\i g(t_j)x_1\right)\psi^{\ell}_M(\cdot)\right\|_{H^1_{g(t)}(\Omega^\Lambda_H)}&\leq C h\|v(t_j,\cdot)\|_{H^2_{g(t)}(\Omega^\Lambda_H)}\\
&\leq C h \|v\|_{C_p^{2n}([a_0,a_1];H^2_{g(t)}(\Omega^\Lambda_H))}.
\end{aligned}
\end{equation*}
For $v_{N,h}=\sum_{j=1}^N\sum_{\ell=1}^M v^{(j,\ell)}_{N,h}\exp\left(-\i g(t_j)x_1\right)\phi_N^{(j)}(t)\psi^{\ell}_M(x)$, and let $v_j^h(x):=v(t_j,x)-\sum_{\ell=1}^M v_{N,h}^{(j,\ell)}\exp\left(-\i g(t_j)x_1\right)\psi^\ell_M(x)$, then $\left\|v_j^h\right\|_{H^1_{g(t)}(\Omega^\Lambda_H)}\leq C h \|v\|_{H_0^{2n}([a_0,a_1];H^2_{g(t)}(\Omega^\Lambda_H))}$. With the result of Lemma \ref{th:basis_fourier},
\begin{equation*}
\begin{aligned}
&\left\|v_N-v_{N,h}\right\|^2_{L^2(\Wast;L^2(\Omega^\Lambda_H))}\\
= &\int_\Wast\int_{\Omega^\Lambda_H}\left|\sum_{j=1}^N\phi_N^{(j)}(t)v(t_j,x)-\sum_{j=1}^N\sum_{\ell=1}^M v^{(j,\ell)}_{N,h}\exp\left(-\i g(t_j)x_1\right)\phi_N^{(j)}(t)\psi^{\ell}_M(x)\right|^2\d x\d t\\
=& \int_\Wast\int_{\Omega^\Lambda_H}\left|\sum_{j=1}^N\phi_N^{(j)}(t)v^h_j(x)\right|^2\d x\d t\\
\leq& \int_\Wast\int_{\Omega^\Lambda_H}\sum_{j,j'=1}^N \phi_N^{(j)}(t)\overline{\phi_N^{(j')}}(t)v^h_j(x)\overline{v^h_{j'}}(x)\d x\d t\\
=& \frac{a_1-a_0}{N}\sum_{j=1}^N\left\|v^h_j\right\|^2_{L^2(\Omega^\Lambda_H)}.
\end{aligned}
\end{equation*}
Similarly, we can have the $H^1$-estimate 
\begin{equation}\label{eq:err2}
\begin{aligned}
\left\|v_N-v_{N,h}\right\|^2_{L^2([a_0,a_1];H^1_{g(t)}(\Omega^\Lambda_H))}&\leq \frac{a_1-a_0}{N}\sum_{j=1}^N\left\|v^h_j\right\|^2_{H^1_{g(t)}(\Omega^\Lambda_H)}\\
&\leq C^2\left[a_0-a_1\right]h^2\|v\|^2_{C_p^{2n}([a_0,a_1];H^2_{g(t)}(\Omega^\Lambda_H))}
\end{aligned}
\end{equation}
With the inequalities \eqref{eq:err1} and \eqref{eq:err2}, 
\high{\begin{equation*}
\begin{aligned}
\left\|v-v_{N,h}\right\|_{L^2(\Wast;H^1_{g(t)}(\Omega^\Lambda_H))}&\leq \left\|v-v_N\right\|_{L^2(\Wast;H^1_{g(t)}(\Omega^\Lambda_H))}+\left\|v_N-v_{N,h}\right\|_{L^2(\Wast;H^1_{g(t)}(\Omega^\Lambda_H))}\\
&\leq {CN^{-2n-1/2}\|v\|_{C^{2n+1}([A_0,A_1];H^1(\Omega^\Lambda_H))}}+Ch \|v\|_{C_p^{2n}([A_0,A_1];H^2(\Omega^\Lambda_H))}\\
&\leq C(N^{-2n-1/2}+h)\|v\|_{C^{2n+1}([A_0,A_1];H^2_{g(t)}(\Omega^\Lambda_H))}.
\end{aligned}
\end{equation*}}
where $C$ depends on $n$ and $g$.
\end{proof}

With this results, we can estimate the error of the approximation of the Bloch transformed total field $w$ by functions in $\widetilde{X}_{N,h}$, when $\Wast=[a_0,a_1]$.
\begin{theorem}\label{th:interp2}
Suppose $u^i$ satisfies Assumption \ref{asp}, $g$ satisfies Assumption \ref{th:asp}, $\zeta,\,\zeta_p\in C^{2,1}(\R)$. \high{$u$ is the total field with the incident field $u^i$ and surface $\Gamma_p$, and  $u_T:=u\circ\Phi_p$, and define $w$ is the Bloch transform of $u_T$, i.e., $w:=\J_\Omega u_T$}. Let $v$ be defined by \high{\eqref{eq:integrand}}, then $v$ satisfies the variational equation \eqref{eq:var_bloch_trans}, for any $\phi\in L^2(\Wast;H^1_\alpha(\Omega^\Lambda_H))$. \high{Moreover}, if $h=O(t^{m+1})$ for some $m\in\N$ as $t\rightarrow0^+$, the approximation of $v$ in $\widetilde{X}_{N,h}$ is bounded by
\begin{equation*}
\min_{v_{N,h}\in \widetilde{X}_{N,h}}\|v-v_{N,h}\|_{L^2(\Wast;H^1(\Omega^\Lambda_H))}\leq C\left(N^{-2m-1/2}+h\right)\high{\|v\|_{C^{2m+1}([A_0,A_1];H^2(\Omega^\Lambda_H))}};
\end{equation*}
if $h=o(t^{n+1})$ for any $n\in\N$ as $t\rightarrow0^+$,  the approximation of $v$ in $\widetilde{X}_{N,h}$ is bounded by
\begin{equation*}
\min_{v_{N,h}\in \widetilde{X}_{N,h}}\|v-v_{N,h}\|_{L^2(\Wast;H^1(\Omega^\Lambda_H))}\leq C\left(N^{-2n-1/2}+h\right)\high{\|v\|_{C^{2n+1}([A_0,A_1];H^2(\Omega^\Lambda_H))}},
\end{equation*}
for any $n\in\N$.
\end{theorem}

\begin{proof}
As $u^i$ satisfies Assumption \ref{asp} and $\zeta,\,\zeta_p\in C^{2,1}(\R)$, from Corollary \ref{th:reg4}, for the interval $\Wast=[a_0,a_1]$, there are three functions $w_0,\,w_1,\,w_2\in C^\infty(\Wast;\widetilde{H}^2_\alpha(\Omega^\Lambda_H))$ such that
\begin{equation*}
w=w_2+\sqrt{\alpha-A_0}\, w_0+\sqrt{\alpha-A_1}\,w_1.
\end{equation*}
So $w$ also satisfies the conditions in Assumption \ref{asp2}. From the definition
\begin{equation*}
v(t,\cdot)=w(g(t),\cdot)g'(t).
\end{equation*}
From Theorem \ref{th:regul_trans}, there are two different cases. 

1) If $h(t)=O(t^{m+1})$ as $t\rightarrow 0^+$ for some $m\in\N$, $v\in C_p^{2m}(\Wast;S(W))\cap C^{2m+1}(\Wast;S(W))$, then from Theorem \ref{th:interp1}, the approximation of $v$ in $\widetilde{X}_{N,h}$ is bounded by
\begin{equation*}
\min_{v_{N,h}\in \widetilde{X}_{N,h}}\left\|v-v_{N,h}\right\|_{\Wast;H^1(\Omega^\Lambda_H))}\leq C(N^{-2m-1/2}+h)\high{\|v\|_{C^{2m+1}([A_0,A_1];H^2(\Omega^\Lambda_H))}}.
\end{equation*}

2) If $h(t)=o(t^{n+1})$ as $t\rightarrow 0^+$ for any $n\in\N$, $v\in C_p^\infty(\Wast;S(W))$, then from Theorem \ref{th:interp1}, the approximation of $v$ in $\widetilde{X}_{N,h}$ is bounded by
\begin{equation*}
\min_{v_{N,h}\in \widetilde{X}_{N,h}}\left\|v-v_{N,h}\right\|_{L^2(\Wast;H^1(\Omega^\Lambda_H))}\leq C(N^{-2n-1/2}+h)\high{\|v\|_{C^{2n+1}([A_0,A_1];H^2(\Omega^\Lambda_H))}}
\end{equation*}
for any $n\in\N$.
\end{proof}

Similar settings and results could be obtained for $\Wast=[a_0,a_2]$, with nodal points $t_j$ are chosen in the same way as Theorem \ref{th:invBloch_num}. Following the proof of Theorem 9 in \cite{Lechl2017}, the error estimate of the finite element method is obtained.
\begin{theorem}
\label{th:FEM}
Assume that $u^i\in H_r^2(\Omega_H)$ for $r\geq 1/2$ satisfies Assumption \ref{asp2}, $\zeta,\,\zeta_p\in C^{2,1}(\R)$, then the linear system \eqref{eq:var_bloch_trans} is uniquely solvable in $\widetilde{X}_{N,h}$ for large enough $N$ and small enough $h>0$. 

1) If $h(t)=O(t^{m+1})$ for some positive integer $m$, then the solution $w_{N,h}\in \widetilde{X}_{N,h}$ satisfies the error estimate
 \begin{equation}
\|v-v_{N,h}\|_{L^2(\Wast;H^\ell(\Omega^\Lambda_H))}\leq C h^{1-\ell}\left(N^{-2m-1/2}+h\right)\high{\|v\|_{C^{2m+1}([A_0,A_1];H^2(\Omega^\Lambda_H))}},\quad \ell=0,1,
\end{equation}
where $C$ depends on $m,u,g$.

2)If $h(t)=o(t^{n+1})$ for any positive integer $n$, then the solution $w_{N,h}\in \widetilde{X}_{N,h}$ satisfies the error estimate
 \begin{equation}
\|v-v_{N,h}\|_{L^2(\Wast;H^\ell(\Omega^\Lambda_H))}\leq C h^{1-\ell}\left(N^{-2n-1/2}+h\right)\high{\|v\|_{C^{2n+1}([A_0,A_1];H^2(\Omega^\Lambda_H))}},\quad \ell=0,1,
\end{equation}
for any $n\in\N$, where $C$ depends on $n,u,g$.
\end{theorem}

\section{Numerical Results}

In this section, we will show some numerical \high{results} with the finite element method introduced in this paper. The periodic function is defined by
\begin{equation*}
\zeta(t)=1.9+\frac{\sin (t)}{3}-\frac{\cos(2t)}{4}
\end{equation*}
with the local perturbation defined by
\begin{equation*}
\zeta_p(t)-\zeta(t)=\begin{cases}
\exp\left(\frac{1}{t^2-1}\right)\sin\left[\pi(t+1)\right],\quad\text{ where }t\in[-1,1];\\
0,\quad\text{ otherwise.}
\end{cases}
\end{equation*}
In all of the numerical examples, we fix $H=4$, and the period $\Lambda=2\pi$, thus $\Lambda^*=1$.

\begin{figure}[htb]
\centering
\begin{tabular}{c c}
\includegraphics[width=0.4\textwidth]{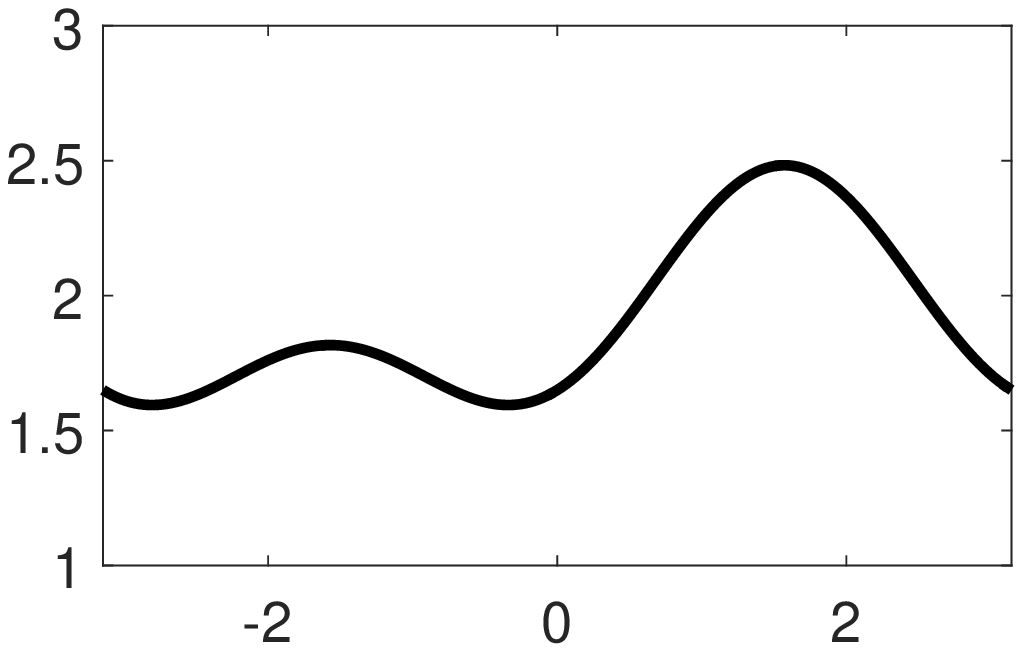} 
& \includegraphics[width=0.4\textwidth]{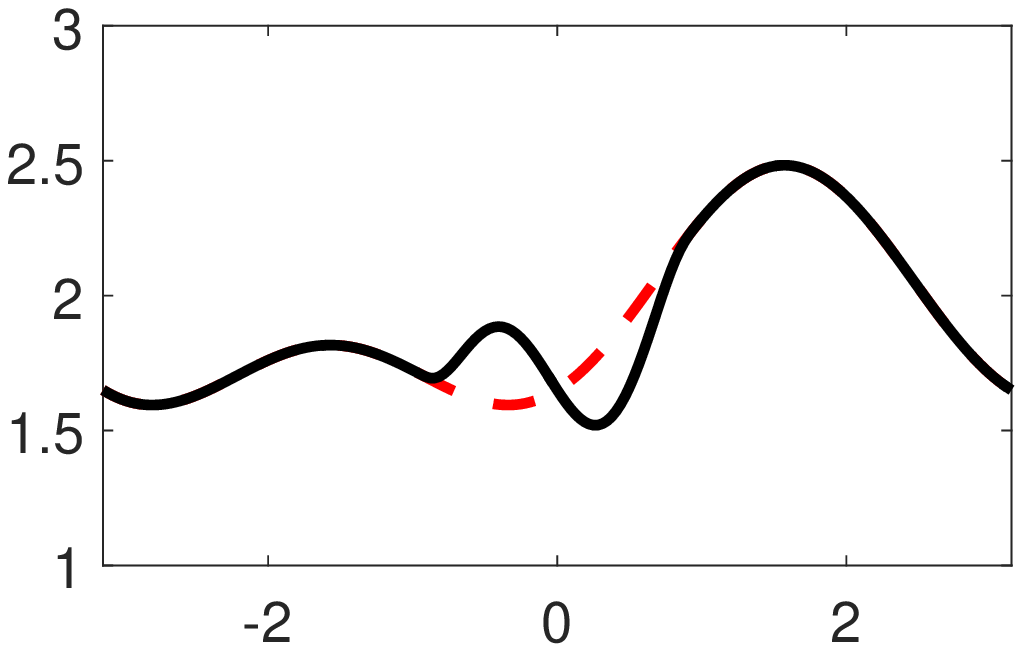}\\[-0cm]
(a) & (b)  
\end{tabular}%
\caption{(a): Periodic surface $\Gamma$; (b): locally perturbed periodic surface $\Gamma_p$.}
\label{fig:surface1}
\end{figure}

We will show the results for two groups. For each group, we consider two different choices of the integration contour, i.e., the function $g$. 

\noindent
{\bf Choice 1.} $g(t)$ is defined in $[A_0,A_1]$ by  
\begin{equation*}
g_1(t)=A_0+\frac{A_1-A_0}{\int_{A_0}^{A_1} (s-A_0)^3(s-A_1)^3\d s}\left[\int_{A_0}^t (s-A_0)^3(s-A_1)^3\d s\right],
\end{equation*}
thus $h(t)=O(t^2)$ as $t\rightarrow 0^+$. 

\noindent
{\bf Choice 2.} $g(t)$ is defined in $[A_0,A_1]$ by 
\begin{equation*}
g_2(t)=A_0+\frac{A_1-A_0}{\int_{A_0}^{A_1} \exp\left(\frac{1}{(s-A_0)(s-A_1)}\right)\d s}\left[\int_{A_0}^t \exp\left(\frac{1}{(s-A_0)(s-A_1)}\right)\d s\right],
\end{equation*}
thus $h(t)=o(t^n)$ as $t\rightarrow 0^+$ for any $n\in\N$.

\begin{figure}[htb]
\centering
\begin{tabular}{c c}
\includegraphics[width=0.4\textwidth]{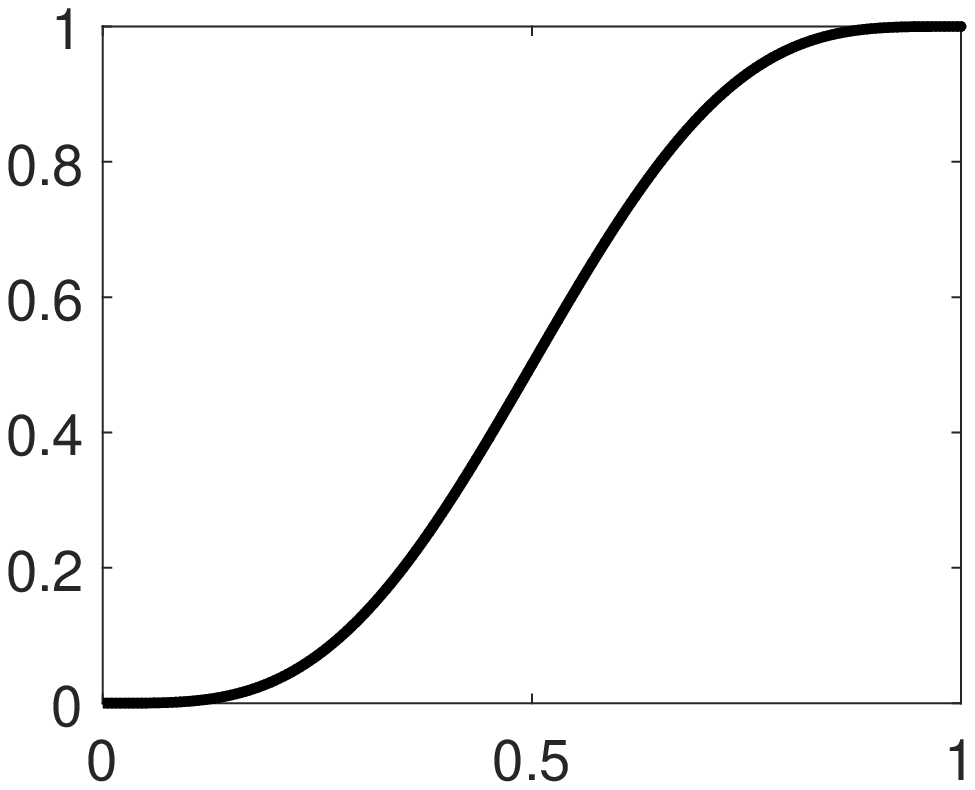} 
& \includegraphics[width=0.4\textwidth]{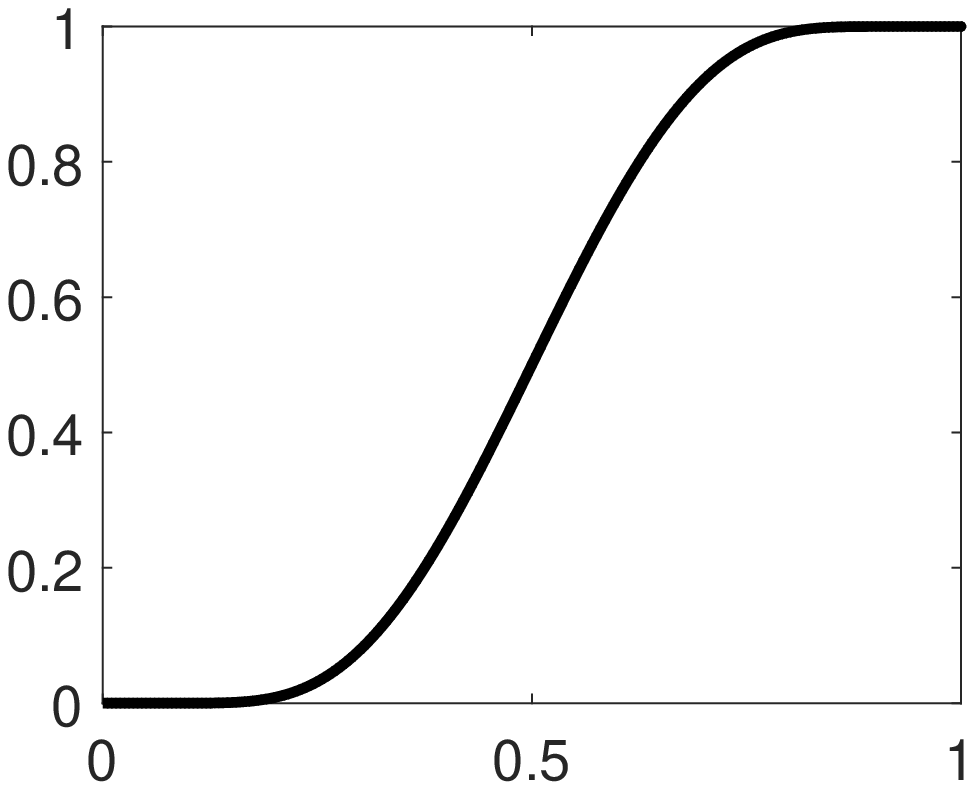}\\[-0cm]
(a) & (b)  
\end{tabular}%
\caption{(a): function $g_1$; (b): function $g_2$.}
\label{fig:surface2}
\end{figure}

Each group contains four different examples. In the four examples, two of them are with non-perturbed surface $\Gamma$ and another two are with locally perturbed surface $\Gamma_p$. The incident field is defined by a half-space Green's function, that is defined as
\begin{equation*}
u^i_G(x,y)=\frac{\i}{4}\left[H_0^{(1)}(k|x-y|)-H_0^{(1)}(k|x-y'|)\right]
\end{equation*}
where $y=(y_1,y_2)^\top$ and $y'=(y_1,-y_2)^\top$, or the upward propagating Herglotz wave function
\begin{equation*}
u^i_+(x)=\int_{-\pi/2}^{\pi/2}\exp\left[\i k(\sin\theta x_1+\cos\theta x_2)\right]\cos^2(\theta)\d\theta,
\end{equation*}
or the downward propagating Herglotz wave function
\begin{equation*}
u^i_-(x)=\int_{-\pi/2}^{\pi/2}\exp\left[\i k(\sin\theta x_1-\cos\theta x_2)\right]\cos^2(\theta)\d\theta.
\end{equation*}

\noindent
{\bf Group 1.} In this group, we show four upward propagating incident fields scattered by both perturbed and non-perturbed surface. \high{As the upward propagating incident fields satisfy the Helmholtz equation above the surface and also satisfy the radiation condition, from the boundary condition that $u^s=-u^i$ and the unique solvability of the direct scattering problem, $u^s=-u^i$ above the (perturbed) surface.} The parameter $h$ is fixed to be $0.01$. The $L^2$-relative errors, defined by $\|u^s_{N,h}-u^s\|_{L^2(\Gamma^\Lambda_H)}/\|u^s\|_{L^2(\Gamma^\Lambda_H)}$, for $N=4,8,16,32,64$ in Table \ref{group1_2} for $g=g_1$ and in Table \ref{group1_infty} for $g=g_2$.
\begin{enumerate}
\item The surface is perturbed. The incident field is the half-space Green's function $u^i_G(\cdot,y)$ with $y=(0.5,0.4)^\top$,  $k=1$.
\item The surface is not perturbed. The incident field is the half-space Green's function $u^i_G(\cdot,y)$ with$y=(-1,0.4)^\top$, $k=1$.
\item The surface is perturbed. The incident field is the  upward propagating Herglotz wave function $u^i_+$, $k=\sqrt{2}$
\item  The surface is not perturbed. The incident field is the upward propagating Herglotz wave function $u^i_+$, $k=1.5$.
\end{enumerate}

\begin{table}[hhhtttttt]
\centering
\begin{tabular}
{|p{2cm}<{\centering}||p{2cm}<{\centering}|p{2cm}<{\centering}
 |p{2cm}<{\centering}|p{2cm}<{\centering}|p{2cm}<{\centering}|p{1.8cm}<{\centering}|}
\hline
  & $Eg\, 1$ & $Eg\, 2$ & $Eg\, 3$ & $Eg\, 4$ \\
\hline
\hline
$N=4$&$7.95$E$-02$&$9.14$E$-02$&$4.76$E$-02$&$4.76$E$-02$\\
\hline
$N=8$&$8.28$E$-04$&$8.43$E$-04$&$5.24$E$-04$&$5.27$E$-04$\\
\hline
$N=16$&$5.78$E$-05$&$3.25$E$-05$&$5.26$E$-05$&$3.20$E$-05$\\
\hline
$N=32$&$2.23$E$-05$&$1.50$E$-05$&$3.75$E$-05$&$2.38$E$-05$\\
\hline
$N=64$&$2.12$E$-05$&$1.50$E$-05$&$3.69$E$-05$&$2.38$E$-05$\\
\hline
\end{tabular}
\caption{Group 1, $g=g_1$.}
\label{group1_2}
\end{table}

\begin{table}[hhhtttttt]
\centering
\begin{tabular}
{|p{2cm}<{\centering}||p{2cm}<{\centering}|p{2cm}<{\centering}
 |p{2cm}<{\centering}|p{2cm}<{\centering}|p{2cm}<{\centering}|p{1.8cm}<{\centering}|}
\hline
  & $Eg\, 1$ & $Eg\, 2$ & $Eg\, 3$ & $Eg\, 4$\\
\hline
\hline
$N=4$&$1.99$E$-01$&$2.22$E$-01$&$1.15$E$-01$&$1.15$E$-01$\\
\hline
$N=8$&$5.03$E$-04$&$6.27$E$-04$&$3.42$E$-04$&$3.43$E$-04$\\
\hline
$N=16$&$2.35$E$-05$&$2.05$E$-05$&$3.75$E$-05$&$2.49$E$-05$\\
\hline
$N=32$&$2.17$E$-05$&$1.49$E$-05$&$3.73$E$-05$&$2.38$E$-05$\\
\hline
$N=64$&$2.13$E$-05$&$1.49$E$-05$&$3.70$E$-05$&$2.38$E$-05$\\
\hline
\end{tabular}
\caption{Group 1, $g=g_2$.}
\label{group1_infty}
\end{table}

\noindent
{\bf Group 2.} In this group, we show four downward propagating incident fields scattered by both perturbed and non-perturbed surface. As the exact value of the scattered field could not be written out, we set the solution with \high{$N=256$} be the "exact value", and show the relative errors for $N=4,8,16,32,64$ in Table \ref{group2_2} for $g=g_1$ and in Table \ref{group2_infty} for $g=g_2$. Figure \ref{fig:err} plots in logarithmic scale of the relative $L^2$-errors for the examples, see (a) for $g=g_1$ and (b) for $g=g_2$. The mesh size $h$ is fixed to be $0.02$.
\begin{enumerate}
\item The surface is perturbed. The incident field is the half-space Green's function $u^i_G(\cdot,y)$ with $y=(0.5,3)^\top$, $k=1$.
\item The surface is non-perturbed. The incident field is the half-space Green's function $u^i_G(\cdot,y)$ with $y=(-1,3)^\top$, $k=1.5$.
\item The surface is perturbed. The incident field is the downward propagating Herglotz wave function $u^i_-$, $k=\sqrt{2}$.
\item The surface is non-perturbed.  The incident field is the downward propagating Herglotz wave function $u^i_-$, $k=2.01$.
\end{enumerate}

\begin{table}[hhhtttttt]
\centering
\begin{tabular}
{|p{2cm}<{\centering}||p{2cm}<{\centering}|p{2cm}<{\centering}
 |p{2cm}<{\centering}|p{2cm}<{\centering}|p{2cm}<{\centering}|p{1.8cm}<{\centering}|}
\hline
  & $Eg\, 1$ & $Eg\, 2$ & $Eg\, 3$ & $Eg\, 4$\\
\hline
\hline
$N=4$&$2.19$E$-01$&$6.12$E$-01$&$1.23$E$-02$&$2.49$E$-02$\\
\hline
$N=8$&$1.71$E$-03$&$3.84$E$-03$&$4.98$E$-04$&$5.62$E$-04$\\
\hline
$N=16$&$9.24$E$-05$&$1.44$E$-04$&$3.17$E$-05$&$3.50$E$-05$\\
\hline
$N=32$&$5.64$E$-06$&$8.72$E$-06$&$1.99$E$-06$&$2.19$E$-06$\\
\hline
$N=64$&$3.49$E$-07$&$5.38$E$-07$&$1.24$E$-07$&$1.36$E$-07$\\
\hline
\end{tabular}
\caption{Group 2, $g=g_1$}
\label{group2_2}
\end{table}

\begin{table}[hhhtttttt]
\centering
\begin{tabular}
{|p{2cm}<{\centering}||p{2cm}<{\centering}|p{2cm}<{\centering}
 |p{2cm}<{\centering}|p{2cm}<{\centering}|p{2cm}<{\centering}|p{1.8cm}<{\centering}|}
\hline
  & $Eg\, 1$ & $Eg\, 2$ & $Eg\, 3$ & $Eg\, 4$\\
\hline
\hline
$N=4$&$4.76$E$-01$&$1.05$E$+00$&$3.24$E$-02$&$6.20$E$-02$\\
\hline
$N=8$&$4.28$E$-03$&$2.34$E$-02$&$9.30$E$-05$&$2.39$E$-04$\\
\hline
$N=16$&$2.30$E$-05$&$3.76$E$-05$&$7.15$E$-06$&$8.05$E$-06$\\
\hline
$N=32$&$2.61$E$-08$&$4.04$E$-08$&$9.20$E$-09$&$1.01$E$-08$\\
\hline
$N=64$&$2.30$E$-12$&$2.90$E$-12$&$6.72$E$-13$&$6.60$E$-13$\\
\hline
\end{tabular}
\caption{Group 2, $g=g_2$}
\label{group2_infty}
\end{table}

\begin{figure}[htb]
\centering
\begin{tabular}{c c}
\includegraphics[width=0.45\textwidth]{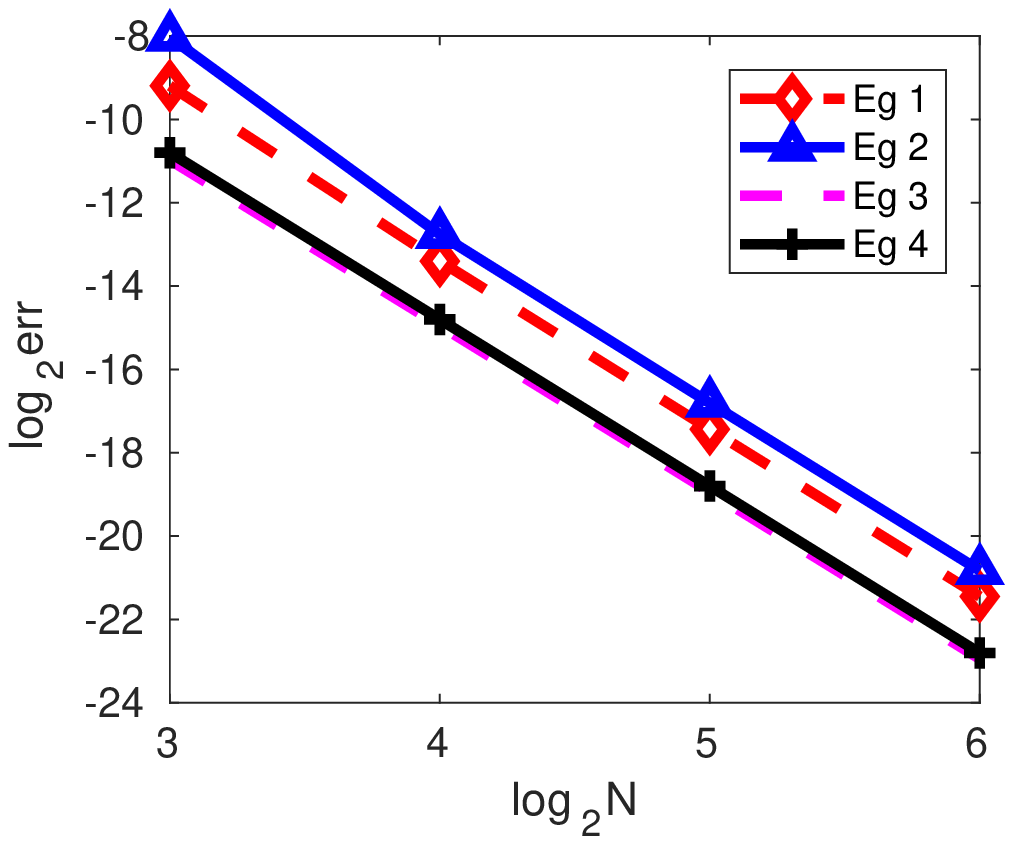} 
& \includegraphics[width=0.45\textwidth]{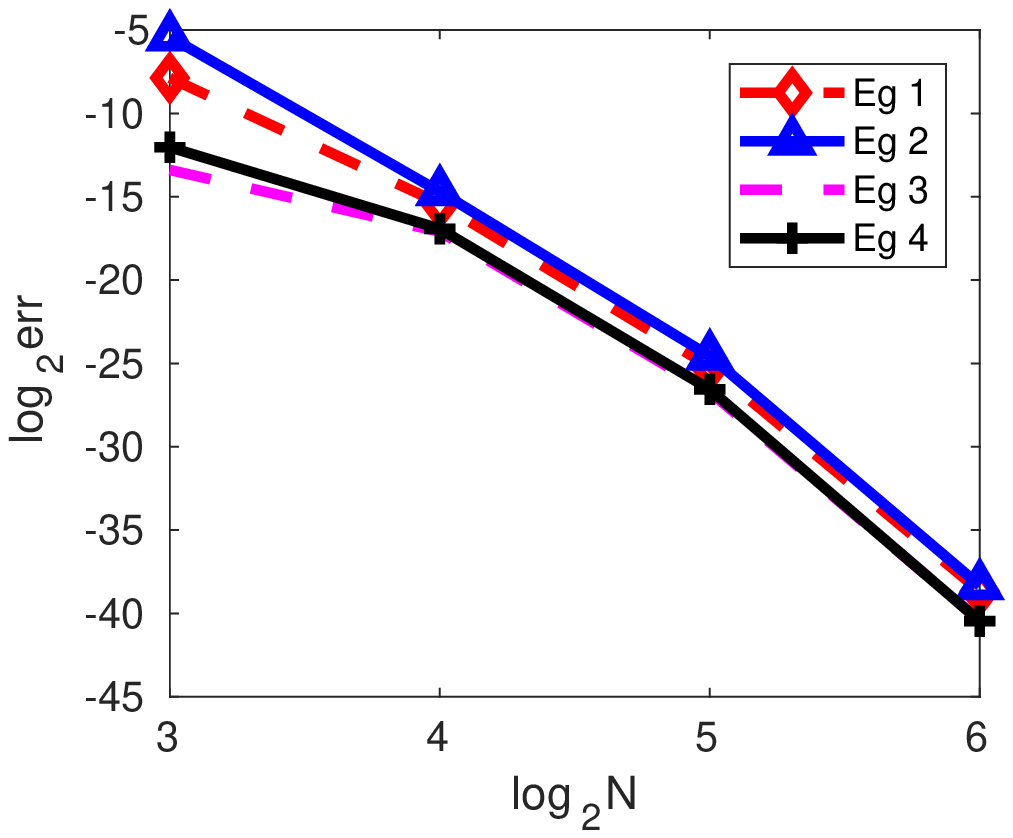}\\[-0cm]
(a) & (b)  
\end{tabular}%
\caption{(a): function $g_1$; (b): function $g_2$.}
\label{fig:err}
\end{figure}

First let's go to the examples in Choice 1. 
In Table \ref{group1_2} and \ref{group1_infty}, the relative error decays fast at first, and then stays at the level of $10^{-5}$ as $N$ gets larger ($N\geq 16$). As we have shown, the $L^2$-relative error is bounded by $C(N^{-2n-1.2}+h)h$, when $h$ is small enough, the error is brought by $N$ is comparatively larger, then when $N$ is relatively small, the error decays fast when $N$ increases. However, when $h$ is not small engough, the dominant part of error is brought by $h$, then the error will stay at a certain level when $N$ increases. Note that the first two examples Group 1 were shown in \cite{Lechl2016a} and \cite{Lechl2017}, as we list in Table \ref{group_old}.
\begin{table}[hhhtttttt]
\centering
\begin{tabular}
{|p{2cm}<{\centering}||p{2cm}<{\centering}|p{2cm}<{\centering}
 |}
\hline
  & $Eg\, 1$ & $Eg\, 2$ \\
\hline
\hline
$N=20$&$7.76$E$-03$&$1.58$E$-02$\\
\hline
$N=40$&$2.83$E$-03$&$5.60$E$-03$\\
\hline
$N=80$&$1.04$E$-03$&$1.98$E$-03$\\
\hline
$N=160$&$4.04$E$-04$&$7.01$E$-04$\\
\hline
$N=320$&$1.93$E$-04$&$2.49$E$-04$\\
\hline
\end{tabular}
\caption{Numerical results in previous papers.}
\label{group_old}
\end{table}

Compared to the first and second columns in Table \ref{group1_2} and Table \ref{group1_infty}, even the result produced by our new method in this  paper at $N=16$ is much better than that produced by the old method at $N=320$. This means during the computational procession, we can save a lot of  time and memory space in setting up the matrix and solving the linear system \ref{eq:linearsystem}.

Now let's go to the examples in Choice 2. 
In Table \ref{group2_2} and \ref{group2_infty}, the error decays as $N$ gets larger, and the logarithmic scale of the $L^2$-errors have been plotted in Figure \ref{fig:err}. 
Figure \ref{fig:err} shows the convergences of the new method for both $g=g_1$ and $g_2$. In (a), the error decreases at the rate that is about $O(N^{-4})$. For $m=1$, the expected error is $O(N^{-2m-1/2}=O(N^{-2-1/2})$, thus the error decays faster than expected. In (b), the errors decrease much more faster, thus it shows the super algebraic convergence rate of the numerical method.  The results in the two tables are good illustrations for the result in Theorem \ref{th:FEM}. With these results shown above, we are confident to say that the new method is very efficient to solve the (locally-perturbed) periodic scattering problems numerically.

\section*{Dedication.} This paper is dedicated to {\em Prof. Dr. Armin Lechleiter}, who passed away in January 2018 at a very young age. He was a talented mathematician, a helpful colleague, and a nice friend. His mathematical theory and numerical analysis for the
scattering problems from locally perturbed periodic surfaces, which have been cited several times, are essential for the work in this paper. The author is very grateful for his scientific contribution to this area and his advise during their collaboration. He will always be missed.

\section*{Acknowledgements.} The author was supported by the University of Bremen and the European Union FP7 COFUND under grant agreement n$^\circ{}\,$600411.

\bibliographystyle{alpha}
\bibliography{ip-biblio} 

\end{document}